\documentclass[12pt]{amsart}
\usepackage{amsmath}
\usepackage{amsfonts}
\usepackage{amssymb}
\usepackage{mathtools}
\usepackage{bbding}
\usepackage{stmaryrd}
\usepackage{txfonts}
\usepackage{graphicx}
\usepackage{epsfig}
\usepackage{xypic}
\usepackage{tikz}
\usepackage{longtable}
\usepackage[all]{xy}
\usepackage{pgflibraryarrows}
\usepackage{pgflibrarysnakes}
\usepackage[shortlabels]{enumitem}
\usepackage{ifpdf}
\ifpdf
  \usepackage[colorlinks,final,backref=page,hyperindex]{hyperref}
\else
  \usepackage[colorlinks,final,backref=page,hyperindex,hypertex]{hyperref}
\fi
\usepackage{graphicx}
\usepackage{epstopdf}
\usepackage{epsfig}
\usepackage{pdfsync}    

\usepackage{xspace}

\newtheorem{theorem}{Theorem}[section]
\newtheorem{lemma}[theorem]{Lemma}
\newtheorem{coro}[theorem]{Corollary}

\newtheorem{prop}[theorem]{Proposition}
\theoremstyle{definition}
\newtheorem{defn}[theorem]{Definition}
\newtheorem{remark}[theorem]{Remark}
\newtheorem{exam}[theorem]{Example}

\newcommand{\nc}{\newcommand}
\newcommand{\delete}[1]{}

\delete{

}

\delete{

}

\nc{\tred}[1]{\textcolor{red}{#1}}
\nc{\tblue}[1]{\textcolor{blue}{#1}} \nc{\tgreen}[1]{\textcolor{green}{#1}} \nc{\tpurple}[1]{\textcolor{purple}{#1}} \nc{\btred}[1]{\textcolor{red}{\bf #1}} \nc{\btblue}[1]{\textcolor{blue}{\bf #1}} \nc{\btgreen}[1]{\textcolor{green}{\bf #1}} \nc{\btpurple}[1]{\textcolor{purple}{\bf #1}}

\newcommand{\efootnote}[1]{}

\nc{\mlabel}[1]{\label{#1}}  
\nc{\mcite}[1]{\cite{#1}}  
\nc{\mref}[1]{\ref{#1}}  
\nc{\meqref}[1]{\eqref{#1}}  
\nc{\mbibitem}[1]{\bibitem{#1}} 

\delete{
\nc{\mlabel}[1]{\label{#1}  
{\hfill \hspace{1cm}{\bf{{\ }\hfill(#1)}}}}
\nc{\mcite}[1]{\cite{#1}{{\bf{{\ }(#1)}}}}  
\nc{\mref}[1]{\ref{#1}{{\bf{{\ }(#1)}}}}  
\nc{\meqref}[1]{\eqref{#1}{{\bf{{\ }(#1)}}}}  
\nc{\mbibitem}[1]{\bibitem[\bf #1]{#1}} 
}

\renewcommand\geq{\geqslant}
\renewcommand\leq{\leqslant}

\renewcommand\bar[1]{\overline{#1}}

\nc{\name}[1]{{\bf #1}}

\nc{\tforall}{\quad \text{ for all }}

\nc{\mre}{\mathrm{Re}\,}

\nc{\mim}{\mathrm{im}\,}
\nc{\nz}{\varepsilon}
\nc{\Id}{\mathrm{Id}}
\nc{\map}[2]{{#2}^{#1}}
\nc{\gp}{B}
\nc{\GSym}{\mathrm{StSym}}
\nc{\QS}{\mathrm{QS}}
\nc{\bwcc}{H_{\DWC}}
\nc{\C}{\mathrm{C}}
\nc{\WC}{\mathrm{WC}}
\nc{\LWC}{\mathrm{LWC}}
\nc{\BC}{\mathrm{BC}}
\nc{\WBC}{\mathrm{WBC}}
\nc{\LWBC}{\mathrm{LWBC}}

\nc{\dwc}{directional weak composition\xspace}
\nc{\dwcs}{directional weak compositions\xspace}
\nc{\DWC}{\mathrm{DWC}}

\nc{\Irr}{\mathrm{Irr}}
\nc{\vx}{\sigma} \nc{\vy}{\tau} \nc{\dvx}{\sigma^{(1)}} \nc{\dvy}{\tau^{(1)}} \nc{\done}{\vep} \nc{\mcitep}[1]{\mcite{#1}} \nc{\wt}{\mathrm{wt}} \nc{\bre}[1]{|#1|} \nc{\mapmonoid}{\frakM} \nc{\disjoint}{\frakM'}
\nc{\ncpoly}[1]{\langle #1\rangle}  
\nc{\mapm}[1]{\lfloor\!|{#1}|\!\rfloor}
\nc{\diff}[1]{{}^\NC\{ #1 \}} \nc{\disj}[1]{\{{#1}\}'} \nc{\mdisj}[1]{\frakM'(#1)} \nc{\brho}{\bar{\rho}} \nc{\om}{\bar{\frakm}} \nc{\frakn}{\mathfrak n} \nc{\ddeg}[1]{^{(#1)}} \nc{\opset}{X} \nc{\genset}{{Z}} \nc{\NC}{\mathrm{{NC}}} \nc{\leaf}{\mathrm{leaf}} \nc{\twig}{\mathrm{twig}} \nc{\fe}{\mathrm{fl}} \nc{\munderline}[1]{#1} \nc{\bo}{o} \nc{\dep}{\mathrm{depth}} \nc{\ofe}{\mathrm{ofl}} \nc{\dfe}{\mathrm{dfe}} \nc{\fex}{\mathrm{fex}} \nc{\dl}{\mathrm{dlex}} \nc{\db}{\mathrm{db}} \nc{\lex}{\mathrm{lex}} \nc{\clex}{\mathrm{clex}} \nc{\dgp}{\mathrm{dgp}} \nc{\dgx}{\mathrm{dgx}} \nc{\br}{\mathrm{br}} \nc{\obd}{\mathrm{odb}} \nc{\ob}{\mathrm{ob}}
\nc{\pie}{\mathrm{PIE}}
\nc{\rbo}{\mathrm{RBO}}
\nc{\supp}{\mathcal{S}}
\nc{\nul}{\mathcal{Z}}

\nc{\bin}[2]{ (_{\stackrel{\scs{#1}}{\scs{#2}}})}  
\nc{\binc}[2]{ \left (\!\! \begin{array}{c} \scs{#1}\\
    \scs{#2} \end{array}\!\! \right )}  
\nc{\bincc}[2]{  \left ( {\scs{#1} \atop
    \vspace{-1cm}\scs{#2}} \right )}  
\nc{\bs}{\bar{S}} \nc{\cosum}{\sqsubset} \nc{\la}{\longrightarrow} \nc{\rar}{\rightarrow} \nc{\dar}{\downarrow} \nc{\dprod}{**} \nc{\dap}[1]{\downarrow \rlap{$\scriptstyle{#1}$}} \nc{\md}[1]{\bar{#1}} \nc{\uap}[1]{\uparrow \rlap{$\scriptstyle{#1}$}} \nc{\defeq}{\stackrel{\rm def}{=}} \nc{\disp}[1]{\displaystyle{#1}} \nc{\dotcup}{\ \displaystyle{\bigcup^\bullet}\ } \nc{\gzeta}{\bar{\zeta}} \nc{\hcm}{\ \hat{,}\ } \nc{\hts}{\hat{\otimes}} \nc{\barot}{{\otimes}} \nc{\free}[1]{\bar{#1}} \nc{\uni}[1]{\tilde{#1}} \nc{\hcirc}{\hat{\circ}} \nc{\leng}{\ell} \nc{\lleft}{[} \nc{\lright}{]} \nc{\lc}{\lfloor} \nc{\rc}{\rfloor}
\nc{\lb}{[} 
\nc{\rb}{]} 
\nc{\curlyl}{\left \{ \begin{array}{c} {} \\ {} \end{array}
    \right.  \!\!\!\!\!\!\!}
\nc{\curlyr}{ \!\!\!\!\!\!\!
    \left. \begin{array}{c} {} \\ {} \end{array}
    \right \} }
\nc{\longmid}{\left | \begin{array}{c} {} \\ {} \end{array}
    \right. \!\!\!\!\!\!\!}
\nc{\onetree}{\bullet} \nc{\ora}[1]{\stackrel{#1}{\rar}}
\nc{\ola}[1]{\stackrel{#1}{\la}}
\nc{\ot}{\otimes} \nc{\mot}{{{\boxtimes\,}}} \nc{\otm}{\overline{\boxtimes}} \nc{\sprod}{\bullet} \nc{\scs}[1]{\scriptstyle{#1}} \nc{\mrm}[1]{{\rm #1}} \nc{\msum}{\sum\limits}
\nc{\margin}[1]{\marginpar{\rm #1}}   
\nc{\dirlim}{\displaystyle{\lim_{\longrightarrow}}\,} \nc{\invlim}{\displaystyle{\lim_{\longleftarrow}}\,} \nc{\mvp}{\vspace{0.3cm}} \nc{\tk}{^{(k)}} \nc{\tp}{^\prime} \nc{\ttp}{^{\prime\prime}} \nc{\svp}{\vspace{2cm}} \nc{\vp}{\vspace{8cm}} \nc{\proofbegin}{\noindent{\bf Proof: }}
\nc{\proofend}{$\blacksquare$ \vspace{0.3cm}}
\nc{\modg}[1]{\!<\!\!{#1}\!\!>}
\nc{\intg}[1]{F_C(#1)} \nc{\lmodg}{\!<\!\!} \nc{\rmodg}{\!\!>\!} \nc{\cpi}{\widehat{\Pi}}
\nc{\sha}{{\mbox{\cyr X}}}  
\nc{\shap}{{\mbox{\cyrs X}}} 
\nc{\shpr}{\diamond}    
\nc{\shp}{\ast} \nc{\shplus}{\shpr^+}
\nc{\shprc}{\shpr_c}    
\nc{\msh}{\ast} \nc{\zprod}{m_0} \nc{\oprod}{m_1} \nc{\vep}{\varepsilon} \nc{\labs}{\mid\!} \nc{\rabs}{\!\mid}
\nc{\astarrow}{\overset{\raisebox{-3pt}{$\ast$}}{\rightarrow}}

\nc{\sqsym}{Stirling quasisymmetric function\xspace}
\nc{\sqsyms}{Stirling quasisymmetric functions\xspace}

\nc{\EEsym}{\mathbb{E}sym}
\nc{\Sym}{\mrm{Sym}}
\nc{\NSym}{\mrm{NSym}}
\nc{\QSym}{\mrm{QSym}}
\nc{\RQSym}{\mrm{RQSym}}
\nc{\RenQSym}{\mrm{RenQSym}}	
\nc{\DQSym}{\mrm{DQSym}}
\nc{\WDQSym}{\mrm{WDQSym}}
\nc{\DLQSym}{\mrm{DLQSym}}
\nc{\ZQSym}{\mrm{ZQSym}}
\nc{\Ensym}{\mrm{ENSym}}
\nc{\Wcsym}{\mrm{WCSym}}
\nc{\LWQSym}{\mrm{LWQSym}}
\nc{\LWCQSym}{\mrm{\mathrm{LWQSym}}}
\nc{\Wcqsym}{\mrm{QSym}_{\widetilde{\mathbb{N}}}}
\nc{\Syms}{symmetric functions\xspace}
\nc{\eqsym}{extended quasisymmetric function\xspace}
\nc{\eqsyms}{extended quasisymmetric functions\xspace}
\nc{\Eqsyms}{Extended Quasisymmetric functions\xspace}
\nc{\Esyms}{Extended symmetric functions\xspace}
\nc{\sgqsym}{quasisymmetric function with semigroup exponents\xspace}
\nc{\sgqsyms}{quasisymmetric functions with semigroup exponents\xspace}
\nc{\Sgqsyms}{Quasisymmetric functions with semigroup exponents\xspace}
\nc{\SGQSYM}{\mrm{SGQSYM}}
\nc{\emzv}{extended multiple zeta value}
\nc{\emzvs}{extended multiple zeta values}
\nc{\sgfps}{formal power series with semigroup exponent\xspace}
\nc{\NSymg}{\mathrm{NSym}_\gp}
\nc{\zqsym}{zeta-quasisymmetric }
\nc{\gslwqsym}{Stirling left weak quasisymmetric function\xspace}
\nc{\gslwqsyms}{Stirling left weak quasisymmetric functions\xspace}
\nc{\ulwb}{upper-left weak bicomposition\xspace}
\nc{\ulwbs}{upper-left weak bicompositions\xspace}

\nc{\parr}{\rm Par}
\nc{\wpar}{\rm WPar}
\nc{\wcomp}{\large{\VDash}}
\nc{\Ker}{\ker}

\nc{\dth}{d} \nc{\mmbox}[1]{\mbox{\ #1\ }} \nc{\fp}{\mrm{FP}} \nc{\rchar}{\mrm{char}} \nc{\Fil}{\mrm{Fil}} \nc{\Mor}{Mor\xspace} \nc{\gmzvs}{gMZV\xspace} \nc{\gmzv}{gMZV\xspace} \nc{\mzv}{MZV\xspace} \nc{\mzvs}{MZVs\xspace}
\nc{\MZV}{\mathrm{MZV}}
\nc{\Hom}{\mrm{Hom}} \nc{\id}{\mrm{id}} \nc{\im}{\mrm{im}} \nc{\incl}{\mrm{incl}}  \nc{\mchar}{\rm char}

\nc{\Alg}{\mathbf{Alg}} \nc{\Bax}{\mathbf{Bax}} \nc{\bff}{\mathbf f} \nc{\bfk}{{\bf k}} \nc{\bfone}{{\bf 1}} \nc{\bfx}{\mathbf x} \nc{\bfy}{\mathbf y}
\nc{\base}[1]{\bfone^{\otimes ({#1}+1)}} 
\nc{\Cat}{\mathbf{Cat}} \delete{}
\nc{\detail}{\marginpar{\bf More detail}
    \noindent{\bf Need more detail!}
    \svp}
\nc{\Int}{\mathbf{Int}} \nc{\Mon}{\mathbf{Mon}}
\nc{\rbtm}{{shuffle }} \nc{\rbto}{{Rota-Baxter }} \nc{\remarks}{\noindent{\bf Remarks: }} \nc{\Rings}{\mathbf{Rings}} \nc{\Sets}{\mathbf{Sets}}
\nc{\balpha}{\mathbf{\alpha}}

\nc{\BA}{{\mathbb A}} \nc{\CC}{{\mathbb C}} \nc{\DD}{{\mathbb D}} \nc{\EE}{{\mathbb E}} \nc{\FF}{{\mathbb F}} \nc{\GG}{{\mathbb G}} \nc{\HH}{{\mathbb H}} \nc{\LL}{{\mathbb L}} \nc{\NN}{{\mathbb N}} \nc{\KK}{{\mathbb K}} \nc{\PP}{{\mathbb P}} \nc{\QQ}{{\mathbb Q}} \nc{\RR}{{\mathbb R}} \nc{\TT}{{\mathbb T}} \nc{\VV}{{\mathbb V}} \nc{\ZZ}{{\mathbb Z}}

\nc{\cala}{{\mathcal A}} \nc{\calc}{{\mathcal C}} \nc{\cald}{{\mathcal D}} \nc{\cale}{{\mathcal E}} \nc{\calf}{{\mathcal F}} \nc{\calg}{{\mathcal G}} \nc{\calh}{{\mathcal H}} \nc{\cali}{{\mathcal I}} \nc{\call}{{\mathcal L}} \nc{\calm}{{\mathcal M}} \nc{\caln}{{\mathcal N}} \nc{\calo}{{\mathcal O}} \nc{\calp}{{\mathcal P}} \nc{\calr}{{\mathcal R}} \nc{\cals}{{\mathcal S}} \nc{\calt}{{\mathcal T}} \nc{\calw}{{\mathcal W}} \nc{\calk}{{\mathcal K}} \nc{\calx}{{\mathcal X}}
\nc{\calz}{{\mathcal Z}}

\nc{\fraka}{{\mathfrak a}} \nc{\frakA}{{\mathfrak A}} \nc{\frakb}{{\mathfrak b}} \nc{\frakB}{{\mathfrak B}}
\nc{\frakc}{{\mathfrak c}}  \nc{\frakD}{{\mathfrak D}}
\nc{\frakH}{{\mathfrak H}}
\nc{\frakh}{{\mathfrak h}} \nc{\frakM}{{\mathfrak M}}
\nc{\frakO}{{\mathfrak O}}
\nc{\frakE}{{\mathfrak E}}
\nc{\bfrakM}{\overline{\frakM}} \nc{\frakm}{{\mathfrak m}} \nc{\frakP}{{\mathfrak P}} \nc{\frakN}{{\mathfrak N}} \nc{\frakp}{{\mathfrak p}} \nc{\frakS}{{\mathfrak S}}
\nc{\frakk}{{\mathfrak k}}
\nc{\frakx}{{\mathfrak x}}
\nc{\frakl}{{\mathfrak l}} \nc{\ox}{\bar{\frakx}} \nc{\frakX}{{\mathfrak X}} \nc{\fraky}{{\mathfrak y}} \nc\dop{\delta}
\nc{\Reduce}{{\rm Red}}

\font\cyr=wncyr10 \font\cyrs=wncyr7
\nc{\redt}[1]{\textcolor{red}{#1}}
\nc{\zb}[1]{\textcolor{purple}{#1}}
\nc{\li}[1]{\textcolor{red}{#1}}
\nc{\lir}[1]{\textcolor{red}{Li:#1}}
\nc{\yu}[1]{\textcolor{blue}{Yu:#1}}

\nc{\wvec}[2]{{\scriptsize{\Big [ \!\!\begin{array}{c} #1 \\ #2 \end{array} \!\! \Big ]}}}

\nc{\bwvec}[2]{\Big(\wvec{#1}{#2}\Big)}

\nc{\jwvec}[2]{{\scriptsize{\Big [ \!\!\begin{array}{cccccccccccccc} #1 \\ #2 \end{array} \!\! \Big ]}}}
\nc{\bjwvec}[2]{\Big(\jwvec{#1}{#2}\Big)}

\begin{document}
\title[Renormalization of quasisymmetric functions]{Renormalization of quasisymmetric functions}

\author{Li Guo}
\address{Department of Mathematics and Computer Science, Rutgers University, Newark, NJ 07102, USA}
\email{liguo@rutgers.edu}

\author{Houyi Yu}
\address{School of Mathematics and Statistics, Southwest University, Chongqing 400715, China}
\email{yuhouyi@swu.edu.cn}

\author{Bin Zhang}
\address{School of Mathematics, Yangtze Center of Mathematics,
	Sichuan University, Chengdu, 610064, China}
\email{zhangbin@scu.edu.cn}

\hyphenpenalty=8000

\date{\today}

\begin{abstract}
As a natural basis of the Hopf algebra of quasisymmetric functions, monomial quasisymmetric functions are formal power series defined from compositions. The same definition applies to left weak compositions, while leads to divergence for other weak compositions. We adapt the method of renormalization in quantum field theory, in the framework of Connes and Kreimer, to deal with such divergency. This approach defines monomial quasisymmetric functions for any weak composition as power series while extending the quasi-shuffle (stuffle) relation satisfied by the usual quasisymmetric functions. The algebra of renormalized quasisymmetric functions thus obtained turns out to be isomorphic to the quasi-shuffle algebra of weak compositions, giving the former a natural Hopf algebra structure and the latter a power series realization. This isomorphism also gives the free commutative Rota-Baxter algebra a power series realization, in support of a suggestion of Rota that Rota-Baxter algebra should provide a broad context for generalizations of symmetric functions.
\end{abstract}

\subjclass[2010]{
05E05,  
81T15,  
16T05,  
17B38,  
11M32,  
16W99,  
11B73  
}

\keywords{quasisymmetric function, renormalization, weak composition, Hopf algebra, Rota-Baxter algebra, quasi-shuffle, Stirling number}

\maketitle

\tableofcontents

\hyphenpenalty=8000 \setcounter{section}{0}


\allowdisplaybreaks

\section{Introduction}\mlabel{sec:int}

Applying the method of renormalization, this paper defines monomial quasisymmetric functions as power series for all weak compositions, yielding a Hopf algebra that extends the Hopf algebra of quasisymmetric functions. In doing so, we obtain a power series realization of the free Rota-Baxter algebra on one generator.

\subsection{Quasisymmetric functions and weak compositions}

The notion of quasisymmetric functions was formally introduced by Gessel~\cite{Ge} in 1984 with its motivation traced back to the work of Stanley~\cite{St1972} on $P$-partitions and of MacMahon~\cite{Mac} on plane partitions. See~\cite{LMW} for a history and general background on quasisymmetric functions.

As the study of quasisymmetric functions became popularized in the 1990s together with other generalizations of symmetric functions~\mcite{NCS,MR}, their importance became evident, through their Hopf algebra structure as the terminal object in the category of combinatorial Hopf algebras~\cite{ABS06} and their diverse applications including in combinatorics, number theory and representation theory. In analog to symmetric functions, much of the significance of quasisymmetric functions is exhibited by their canonical linear bases, including the bases of monomial, fundamental and Schur quasisymmetric functions \cite{HLMW11}, all parameterized by {\bf compositions}, that is, vectors of positive integers.

Most natural among these bases is the basis of monomial quasisymmetric functions
\begin{equation} \mlabel{eq:mqsym}
M_\alpha:=\sum_{0<i_1<i_2<\cdots<i_k}x_{i_1}^{\alpha_1}x_{i_2}^{\alpha_2}\cdots x_{i_k}^{\alpha_k}
\end{equation}
for compositions $\alpha=(\alpha_1,\cdots,\alpha_k)$.
By specializing $x_i$ to $1/i$, $M_\alpha$ gives the multiple zeta value (MZV)
\begin{equation} \mlabel{eq:mzv}
\zeta(\alpha):=\sum_{ 0<i_1<i_2<\cdots<i_k} \frac{1}{i_1^{\alpha_1}i_2^{\alpha_2}\cdots i_k^{\alpha_k}},
\end{equation}
which is convergent for $\alpha_k\geq 2, \alpha_j\geq 1, 1\leq j\leq k-1$. This has led to the algebraic study of MZVs which has its origin in the work of Euler and Goldbacher in the two variable case and became highly important since the general study was started in the 1990s~\cite{Br12,Ho92,IKZ,Za,Za2}. On the other hand, monomial quasisymmetric functions are a special case of the quasi-shuffle algebra when the spanning algebra is $x\bfk[x]$ of polynomials in one variable {\em without constant terms}. The quasi-shuffle product has its equivalent forms in terms of stuffles, the sticky shuffle product, mixable shuffle product and overlapping product, among others~\cite{Car1972,EG2006,Gub,Ha,Ho00,Losh}. Under these guises, they appeared in many studies, including Rota-Baxter algebras, Zinbiel algebras and motivic multiple zeta values.

When $\alpha$ is a {\bf weak composition}, that is, a vector of nonnegative integers, the expression of $M_\alpha$ in Eq.~\meqref{eq:mqsym} for the monomial quasisymmetric functions might not be well defined, as indicated by the simple exmaple
\begin{equation}\mlabel{eq:M0ns}
M_{(0)}=\sum_{i=1}^\infty i^0=\sum_{i=1}^\infty 1.
\end{equation}
In fact, it can be easily shown (Lemma~\mref{lem:wccong}) that such an expression gives a well-defined power series precisely when $\alpha$ is left weak, that is, when $\alpha_k$ is nonzero.

Thus it is natural to explore how to make sense of the expression $M_\alpha$ in Eq.~\meqref{eq:mqsym} as a power series for any weak composition $\alpha$.
Our interest in doing this is further motivated by the following considerations.

\begin{enumerate}
\item
The specializations of such expressions to MZVs, though still divergent, have been defined through various renormalization processes~\cite{EMS,GZ08,IKZ,MP}. The success in applying the renormalization method to MZVs suggests that the divergent expressions in Eq.~\meqref{eq:mqsym} be similarly treated, providing another testing ground in applying the renormalization method to divergencies in mathematics;
\item
Quasisymmetric functions have appeared as the main building block for free commutative Rota-Baxter algebras~\mcite{EG2006,GK}, in line with the program proposed by Rota~\mcite{Ro2} that Rota-Baxter algebras provide a natural and broad context for generalizations of symmetric functions.\footnote{To quote Rota, ``Baxter algebras represent the ultimate and most natural generalization of the algebra of symmetric functions."} To follow this program further, one expects that the full free commutative Rota-Baxter algebra be realized as some generalized quasisymmetric functions. This needs a notion of quasisymmetric functions for weak compositions. 
\item
As in the case of symmetric and quasisymmetric functions, it is desirable to obtain power series realizations for other abstractly defined combinatorial Hopf algebras, especially for quasi-shuffle Hopf algebras. See ~\mcite{FNT14,Ma13} for power series realizations in some other cases. Quasisymmetric functions for compositions gives a power series realization of the quasi-shuffle algebra on $x\bfk[x]$. A suitably defined quasisymmetric functions for weak compositions should give a power series realization of the quasi-shuffle algebra on $\bfk[x]$. An intermediate step is taken in~\mcite{GTY19} which gives such a realization in abstract power series with semigroup exponents, which are not the usual power series. The fact that a renormalization process \'a la Connes and Kreimer takes place in Laurent series and power series makes the process a promising channel for power series realizations. 
\end{enumerate}

\subsection{The approach of renormalization}

In this work, we define a power series as the monomial quasisymmetric function for any weak composition, following the algebraic approach of Connes and Kreimer~\mcite{CK2000,EGK} to renormalization in perturbative quantum field theory. Fundamental in their approach is the Algebraic Birkhoff Factorization (ABF) built on Hopf algebras and Rota-Baxter algebras. Such a setup allows the ideas of renormalization in quantum field theory to be applied to extract finite quantities from divergencies in mathematics.

The Algebraic Birkhoff Factorization (in its opposite form; see Theorem~\mref{thm:Abdt}) states that, for a given triple $(H, R, \phi)$ consisting of
\begin{itemize}
\item a connected filtered Hopf algebra $H$,
\item a commutative Rota-Baxter algebra $R$ on which the Rota-Baxter operator $P:R\to R$ is idempotent, and
\item an algebra homomorphism $\phi: H\to R$ serving as the regularization map,
\end{itemize}
there are unique algebra homomorphisms
$
\phi_-: H\to \bfk+P(R)$ and
$\phi_+: H\to \bfk+(\id-P)(R)$
such that
\begin{equation}
\phi=\phi_+\star\phi_-^{\star (-1)}.
\mlabel{eq:abd}
\end{equation}
Here $\star$ is the convolution product and $\phi_+$ is called the {\bf renormalization} of $\phi$.

To apply the Algebraic Birkhoff Factorization to our situation, we construct a Hopf algebra of bicompositions $\wvec{\alpha}{\beta}$ serving as parameters of the regularized quasisymmetric functions of weak compositions, with the second composition $\beta$ serving as directional parameters. We then regularize a monomial quasisymmetric function $M_\alpha$ in Eq.~\meqref{eq:mqsym}, which might a priori be divergent, by perturbing it by a directional controlling factor. Then a renormalization process is carried out by implementing the Algebraic Birkhoff Factorization. Further analysis gives rise to a well-defined power series that will be called the \name{(renormalized) quasisymmetric functions of weak compositions}. These power series have the following properties.

\begin{enumerate}
\item
Distinct weak compositions give linearly independent renormalized quasisymmetric functions. Renormalized quasisymmetric functions coincide with the existing quasisymmetric functions when the latter are already defined. Similar to the case of MZVs where the renormalized MZVs $\zeta(\alpha)$ for positive $\alpha$ (that is, compositions $\alpha$) is a one variable extension over the algebra of convergent MZVs (when $s_k\geq 2$), the renormalized quasisymmetric functions is a one variable extension over the algebra of convergent quasisymmetric functions (that is, for left weak compositions). This similarity suggests a relationship between renormalized quasisymmetric functions and renormalized MZVs further along Eqs.~\meqref{eq:mqsym} and~\meqref{eq:mzv};
\item
Through the connection of quasi-shuffle algebras with free Rota-Baxter algebras, the free commutative Rota-Baxter algebra on one generator is realized as a polynomial extension of the renormalized quasisymmetric functions. We thus get another instance in the direction of Rota's program;
\item
The renormalized quasisymmetric functions are power series forming an algebra naturally isomorphic to the quasi-shuffle algebra of weak compositions, thus giving the former algebra a Hopf algebra structure and the latter algebra a power series realization. It further has the Hopf algebra of quasisymmetric function as both a Hopf subalgebra and a Hopf quotient algebra. 
\end{enumerate}

We next provide further details of our approach which also serve as an outline of the paper.

\subsection{Outline of the paper}

In Section~\mref{sec:back}, we first recall the general principle of Algebraic Birkhoff Factorization of Connes and Kreimer for renormalization.
We then introduce the two main building blocks in order to apply the Algebraic Birkhoff Factorization to divergence in quasisymmetric functions. One is the Hopf algebra of weak bicompositions, serving as the parameter space of quasisymmetric functions of weak compositions with a directional purterbation, and the other is a Rota-Baxter algebra hosting the regularizations of the quasisymmetric functions of weak compositions.

The algebra homomorphism between the two algebras that serves as the regularization process is given in Section~\mref{sec:reg} (Theorem~\mref{thm:phi(lastpartlt0)}), with a careful analysis to ensure that the regularization indeed takes its values in the desired Rota-Baxter algebra.

In Section~\mref{sec:renorm}, first the principle of Algebraic Birkhoff Factorization is applied to the regularization in Section~\mref{sec:reg} to obtain a renormalization of the quasisymmetric function for a weak composition, which still inherits the directional dependency from the regularization. Then this dependency is removed after an averaging process, giving rise to the notion of the renormalized monomial quasisymmetric function uniquely defined for any given weak composition (Theorem~\mref{thm:mainthmwildms}). As expected, the renormalization power series agrees with a quasisymmetric function when it is already a well-defined power series.

Section~\mref{sec:reln} presents properties of the renormalized quasisymmetric functions.
First, the renormalized monomial quasisymmetric functions are shown to extend the quasi-shuffle product of monomial quasisymmetric functions (Theorem~\mref{thm:widetildeMsquasishuffle}).
It is then shown that the algebra of renormalized quasisymmetric functions
is a polynomial algebra of one variable over the ring of left weak quasisymmetric functions (Theorem \mref{thm:rqflmcasymm0}), and is isomorphic to the quasi-shuffle algebra of weak compositions (Thereom~\mref{prop:nalpisakbasis}). Then this algebra automatically inherits a Hopf algebra structure from the quasi-shuffle algebra, having the Hopf algebra of quasisymmetric functions as a Hopf subalgebra and Hopf quotient algebra (Proposition~\mref{pp:subquot}). Furthermore, by a scalar extension, we obtain the free commutative Rota-Baxter algebra of weight $1$ on one generator, giving a realization of the Rota-Baxter algebra as a generalization of symmetric functions, as envisioned by Rota~\mcite{Ro2}. 

\smallskip

\noindent
{\bf Convention.} In this paper we take $\bfk$ to be a commutative unitary ring containing $\QQ$, unless otherwise specified. All algebras, modules, linear maps and tensor products are taken over $\bfk$. Let $\NN$ and $\PP$ denote the set of nonnegative and positive integers respectively.
Given any $m, n \in \PP$ with $m\leq n$, denote $[m,n]:=\{m, m +1,\cdots, n\}$ and $[n]:=[1, n]$.

\section {Algebraic Birkhoff Factorization and weak bicompositions}
\mlabel{sec:back}

This section provides the setup for our renormalization of quasisymmetric functions of weak compositions. We first recall the general principle of Algebraic Birkhoff Factorization in the algebraic approach of Connes and Kreimer~\mcite{CK2000} to renormalization of perterbative quantum field theory. We then define Hopf algebras from various classes of weak compositions and weak bicompositions.
We finally recall bimonomial quasisymmetric functions~\mcite{ABB06} and introduce the notion of \gslwqsyms.

\subsection{Algebraic Birkhoff Factorization}

The Algebraic Birkhoff Factorization is built on the notions of Hopf algebras and Rota-Baxter algebras.

A \name{connected filtered Hopf algebra} \cite{Man} is a Hopf algebra $(H,m,u,\Delta,\vep,S)$ with its product $m$, unit $u:\bfk\to H$, coproduct $\Delta:H\to H\ot H$, counit $\vep:H\to \bfk$ and antipode $S:H\to H$, together with submodules $H^{(n)}, n\geq 0,$ satisfying
\begin{align*}
H^{(0)}=\bfk,\quad H^{(n)}\subseteq H^{(n+1)}, \quad H^{(p)}H^{(q)}\subseteq H^{(p+q)},
\quad \Delta(H^{(n)})\subseteq\sum_{p+q=n}H^{(p)}\otimes H^{(q)},\quad
S(H^{(n)})\subseteq H^{(n)}.
\end{align*}

Fix $\lambda\in\bfk$. A {\bf Rota-Baxter algebra} of weight $\lambda$~\mcite{Ba,Gub} is a pair $(R,P)$ consisting of an algebra $R$ and a linear operator $P:R\rightarrow R$
satisfying the Rota-Baxter identity
\begin{align}\mlabel{eq:RBOidentity}
P(x)P(y)=P(xP(y))+P(P(x)y)+\lambda P(xy) \tforall x,y\in R.
\end{align}
It follows from the definition that $P(R)$ and $(-\lambda\id-P)(R)$ are non-unitary subalgebras of $R$.
Thus $\bfk+P(R)$ and  $\bfk+(-\lambda\id-P)(R)$ are unitary subalgebras.

For an algebra $A$, the Laurent series ring
$$
A[z^{-1},z]]=\Big\{\sum_{n=k}^{\infty}a_nz^n\,\Big|\,k\in\mathbb{Z},a_n\in A,k\leq n<\infty\Big\},
$$
equipped with the projection $P$ onto the polar part, that is,
\begin{align}
P\Big(\sum_{n=k}^{\infty}a_nz^n\Big):=\sum_{n=k}^{-1}a_nz^n, \mlabel{eq:defRBO1}
\end{align}
is a Rota-Baxter algebra of weight $-1$. Furthermore, $P^2=P$.

In the Connes-Kreimer approach to renormalization of perturbative quantum field theory, the principle of renormalization can be formulated algebraically as follows.
\begin{theorem}\mlabel{thm:Abdt}{\rm {(Algebraic Birkhoff Factorization} \cite{CK2000,GZ08,Man})}
	Let $H$ be a connected filtered Hopf algebra, $(R,P)$ a commutative Rota-Baxter algebra of weight $-1$ with $P$ idempotent, and $\phi:H\rightarrow R$ an algebra homomorphism.
	Denote $\check{P}=-P$ and $\tilde{P}=\id-P$.
Then there are unique algebra homomorphisms $\phi_{-}:H\rightarrow \bfk+P(R)$ and $\phi_{+}:H\to \bfk+(\id-P)(R)$ with the factorization
		\begin{align}\mlabel{eq:ABDofphi}
		\phi=\phi _+\star \phi_{-}^{\star(-1)},
		\end{align}
		called the {\bf Algebraic Birkhoff Factorization} of $\phi$.
		Here $\phi_{-}^{\star(-1)}$ is the inverse of $\phi_{-}$ with respect to the convolution product $\star$ on
		the space of linear maps from $H$ to $R$ associated with the coproduct on $H$.
		Further,
		\begin{align}\mlabel{eq:phi-xabd}
		 \phi_{-}(x)=\check{P} \Big(\phi(x)+\sum_{(x)}\phi (x') \phi_{-}(x'') \Big)
		\end{align}
		and
		\begin{align}\mlabel{eq:phi+xabd}
		 \phi_{+}(x)=\tilde{P} \Big(\phi(x)+\sum_{(x)}\phi (x')\phi_{-}(x'') \Big).
		\end{align}
		Here we have used the notation $\Delta(x)=1\otimes x+x\otimes 1+\sum\limits_{(x)}x'\otimes x''$.
\end{theorem}

\begin {remark} The original form of the Algebraic Birkhoff Factorization is
$
\phi=\phi_{-}^{\star(-1)} \star \phi _+.
$
It is equivalent to the form
$
		\phi=\phi _+\star \phi_{-}^{\star(-1)}
$
stated above once the coproduct $\Delta(x)=\sum\limits_{(x)}x_{(1)} \ot x_{(2)}$ is replace by its opposite $\Delta^{\text{op}}(x):=\sum\limits_{(x)}x_{(2)} \ot x_{(1)}.$
\end{remark}

\subsection{Hopf algebras of weak bicompositions}
\mlabel{ss:alg}
After a brief summary on weak compositions and bicompositions, we construct a commutative
connected filtered Hopf algebra of weak bicompositions. 

We first recall basic notions on compositions and weak compositions.
\begin {itemize}
\item A \name{weak composition} $\alpha$ of a nonnegative integer $n$ is a finite ordered list of nonnegative integers $(\alpha_1,\alpha_2,\cdots,\alpha_k)$
such that $\alpha_1+\alpha_2+\cdots+\alpha_k=n$.
The $\alpha_i$ are called the \name{components} of $\alpha$. To simplify the notation, we abbreviate $m$ consecutive components $j$ by $j^m$. 
\item A weak composition is called a \name{left weak composition} if its last component is positive, and is called a \name {composition} if all of its components are positive. Without referring to it sum, a weak composition is simply a vector of nonnegative integers and a composition is a vector of positive integers.
\item Given a weak composition $\alpha= (\alpha_1,\alpha_2,\cdots,\alpha_k)$,
its \name{$0$-length} $\ell_{0}(\alpha)$ is the number of $0$ components of $\alpha$. Its \name{length} $\ell(\alpha)$ is the number of components of $\alpha$.
We also denote its \name{size} by $|\alpha|:=\alpha_1+\alpha_2+\cdots+\alpha_k$, its
\name{total size} by $||\alpha||:=|\alpha|+\ell_{0}(\alpha)$. If $\alpha$ is a composition of size $n$, then we denote $\alpha\models n$.
For convenience we define the \name{empty composition}, denoted $\emptyset$, to be the unique weak composition whose size and length are $0$.
\end{itemize}

Given a nonnegative integer $n$, let $\C(n)$, ${\WC}(n)$ and ${\LWC}(n)$ denote the set of compositions,
weak compositions and left weak compositions of $n$, respectively.
We also write ${\C}$, ${\WC}$ and ${\LWC}$ for the sets of compositions, weak compositions and left weak compositions, respectively. These sets are all graded sets with grading given by the size $n$.
Clearly, we have ${\C}\subsetneq{\LWC}\subsetneq {{\WC}}$.

There are several operations on compositions and weak compositions:
\begin {itemize}
\item The {\bf reversal} of a composition $\alpha$, denoted by $\alpha^{r}$, is obtained by writing the components of $\alpha$ in the reverse order.
\item
The composition obtained from a weak composition $\alpha$ by omitting its $0$ components is denote by $\overline{\alpha}$.
\item For a pair of weak  compositions $\alpha=(\alpha_1,\cdots,\alpha_k)$ and $\beta=(\beta_1,\cdots,\beta_{\ell})$,
the \name{concatenation} of $\alpha$ and $\beta$ is
$\alpha\cdot \beta:=(\alpha_1,\cdots,\alpha_k,\beta_1,\cdots,\beta_\ell)$.
\end {itemize}
For example, take $\alpha=(3,0,1,0)$ and $\beta=(2,0)$. Then $\alpha^r=(0,1,0,3)$, $\overline{\alpha}=(3,1)$ and $\alpha\cdot\beta=(3,0,1,0,2,0)$.

Let $\alpha=(\alpha_1,\alpha_2,\cdots,\alpha_k)$ and $\beta=(\beta_1,\beta_2,\cdots,\beta_k)$ be weak compositions of the same length.
The matrix
\begin{align*}
\wvec{\alpha}{\beta}=\wvec{\alpha_1, \alpha_2, \cdots, \alpha_k}{\beta_1,\beta_2,\cdots,\beta_k}
\end{align*}
is called a {\bf weak bicomposition}. We call $\wvec{\alpha}{\beta}$ an {\bf \ulwb} if $\alpha$ is a left weak composition, that is, if $\alpha_k>0$.
The {\bf empty weak bicomposition}, of length $0$, is also denoted by $\emptyset$.
We denote by $\WBC$ and $\LWBC$ for the sets of weak bicompositions and \ulwbs, respectively.

Let $A$ be an additive semigroup. For each nonnegative integer $n$, let $\bfk A^n$ be the free $\bfk$-module spanned by the Cartesian power $A^n$ with the convention that $A^0=\{\emptyset\}$ and $\bfk A^0=\bfk$.
The \name{quasi-shuffle Hopf algebra}~\cite{GZ08,Ho00} $\QS(A)$ on the semigroup algebra $\bfk A$ is the $\bfk$-module $\bfk\langle A\rangle=\oplus_{n=0}^{\infty}\bfk A^n$ equipped with
\begin{enumerate}
\item the \name{quasi-shuffle product} $*$ with $\emptyset$ the identity element, and satisfying the recursive relation
\begin{align*}
a*b=(a_1,(a_2,\cdots,a_m)*b)
 +(b_1,a*(b_2,\cdots,b_n))+(a_1+b_1,(a_2,\cdots,a_m)*(b_2,\cdots,b_n)),
\end{align*}
where $a=(a_1,a_2,\cdots,a_m)\in A^m$ and $b=(b_1,b_2,\cdots,b_n)\in A^n$;
\item
the deconcatenation coproduct
\begin{align*}
\Delta (a)=\sum_{i=0}^m(a_1,\cdots,a_i)\otimes(a_{i+1},\cdots,a_m)
\tforall a=(a_1,a_2,\cdots,a_m)\in A^m;
\end{align*}
\item
the counit
$$\epsilon:\bfk\langle A\rangle\rightarrow\bfk,\quad \epsilon(a)=\delta_{a,\emptyset} \tforall a\in A^m;$$
\item
the submodules $\bfk\langle A\rangle^{(n)}=\oplus_{i=0}^n\bfk A^i$.
\end{enumerate}

Taking $A$ to be the additive semigroup $\PP$ (resp. the additive monoid $\NN$), we obtain the quasi-shuffle Hopf algebra $\QS(\PP)=\bfk \C$ (resp. $\QS(\NN)=\bfk \WC$) of compositions (resp. weak compositions).
Similarly, taking $A$ to be the additive semigroup $\NN\times \NN=\Big\{\left.\wvec{s}{r}\,\right|\,s\in \NN,r\in \NN\Big\}$ with the vertical notation and the componentwise addition $\wvec{s}{r}+\wvec{s'}{r'}=\wvec{s+s'}{r+r'}$, we obtain the connected filtered Hopf algebra $\QS(\NN\times \NN)=\bfk \WBC$ of weak bicompositions.

Most useful to us is the quasi-shuffle Hopf algebra
$$\bwcc:=\QS(\NN\times \PP),$$
called the \name{Hopf algebra of \dwcs}.
It has a canonical basis
\begin{equation} \mlabel{eq:mbasis}
\DWC:=\Big\{\left . \wvec {\alpha}{\beta}\,\right|\,\alpha\in\mathbb{N}^k,\beta\in\mathbb{P}^k,k\in\mathbb{N}\Big\}.
\end{equation}
whose elements will be called the \name{\dwcs}, regarded as the weak composition $\alpha$ in direction $\beta$. The role of the \dwcs is to represent a perturbation (or regularization) of weak composition quasisymmetric function for $\alpha$ in the direction $\beta$. See Eq.~\meqref{eq:mapphia}. Thus the \dwcs play a role similar to the Feynman diagrams representing regularized Feynman integrals in the original work of Connes and Kreimer~\mcite{CK2000}.

\subsection{Algebras from weak bicompositions}

Let $X:=\{x_i|i\in\PP\}$ and $Y:=\{y_i|i\in\PP\}$ be two sets of mutually disjoint and commuting variables.
Given a weak bicomposition $\wvec{\alpha}{\beta}=\wvec{\alpha_1, \alpha_2, \cdots, \alpha_k}{\beta_1,\beta_2,\cdots,\beta_k}$ with
$\wvec{\alpha_i}{\beta_i}\neq \wvec{0}{0}$ for all $i=1,2,\cdots,k$,
the \name{bimonomial quasisymmetric function} indexed by $\wvec{\alpha}{\beta}$ is
\begin{equation} \mlabel{eq:bimqsym}
M_{\wvec{\alpha}{\beta}}(X,Y):=\sum_{i_1<i_2<\cdots<i_k}x_{i_1}^{\alpha_1}x_{i_2}^{\alpha_2}\cdots x_{i_k}^{\alpha_k}y_{i_1}^{\beta_1}y_{i_2}^{\beta_2}\cdots y_{i_k}^{\beta_k}.
\end{equation}
The notion was introduced by J.-C Aval, F. Bergeron and N. Bergeron in \cite{ABB06} as a more general context for quasisymmetric functions.

Let $\alpha=(\alpha_1,\cdots,\alpha_k)$ be a composition of length $k$ and $I:=\{i_1<\cdots<i_k\}$ be an ordered $k$-subset of $\PP$. We will use the abbreviations 
\begin{equation}\mlabel{eq:abb}
x_I^\alpha:=x_{i_1}^{\alpha_1}\cdots x_{i_k}^{\alpha_k},\quad  I^\alpha:=i_1^{\alpha_1}\cdots i_k^{\alpha_k}, \quad \sum_{I}a_Ix_I^\alpha:=\sum_{i_1<\cdots<i_k} a_{i_1,\cdots,i_k} x_{i_1}^{\alpha_1}\cdots x_{i_k}^{\alpha_k}
\end{equation}
for $a_{i_1,\cdots,i_k}\in \bfk$. Here the sum runs over all ordered $k$-subsets $I$ of $\PP$. Then
$M_{\wvec{\alpha}{\beta}}=\sum\limits_{I} x_I^\alpha y_I^\beta.$

\begin{prop} \mlabel{pp:abb}
$($\cite{ABB06}$)$ The multiplication of two bimonomial quasisymmetric functions $M_{\wvec{\alpha}{\beta}}$ in Eq.~\meqref{eq:bimqsym} satisfy the quasi-shuffle relation:
	$$M_{\wvec{\alpha}{\beta}}(X,Y) M_{\wvec{\alpha'}{\beta'}}(X,Y)=M_{\wvec{\alpha}{\beta}\ast \wvec{\alpha'}{\beta'}}(X,Y)$$
	for the quasi-shuffle product $\wvec{\alpha}{\beta}\ast \wvec{\alpha'}{\beta'}$ of bicompositions. Here for a linear combination $\sum\limits_{\gamma\in \WBC} a_{\gamma} \gamma$ of bicompositions, we use the notation 
	$$M_{\sum\limits_{\gamma\in \WBC} a_{\gamma} \gamma}(X,Y):= \sum_{\gamma\in \WBC}a_{\gamma} M_{\gamma}(X,Y).$$
\end{prop}

For our purpose, we introduce another class of formal power series.

\begin{defn} \mlabel{de:stqsym}
For an \ulwb $\wvec{\alpha}{\beta}$, define the {\bf \gslwqsym} of $\wvec{\alpha}{\beta}$ to be
\footnote{The term Stirling is used since the coefficients in the expansion of $\widehat{M}_{\wvec{\alpha}{\beta}}$ in terms of monomial left weak quasisymmetric functions involve Stirling numbers of the second kind, see Remark~\mref{rk:cstir} and Proposition \mref{prop:transfMtohatM}.}
\begin{equation}\mlabel{eq:sqsym1}
\widehat{M}_{\wvec{\alpha}{\beta}}:=\sum_{I} I^\beta x_I^\alpha=\sum_{i_1<i_2<\cdots<i_k}i_1^{\beta_1}i_2^{\beta_2}\cdots i_k^{\beta_k}x_{i_1}^{\alpha_1}x_{i_2}^{\alpha_2}\cdots x_{i_k}^{\alpha_k}.
\end{equation}
Let $\GSym$
denote the $\QQ$-submodule of $\QQ[[X]]$ spanned by all $\widehat{M}_{\wvec{\alpha}{\beta}}$ with $\wvec{\alpha}{\beta}$ upper-left weak.
\end{defn}

Note that the expression for bimonomial quasisymmetric functions
$M_{\wvec{\alpha}{\beta}}(X,Y)$ in Eq. \eqref{eq:bimqsym} still make sense for an \ulwb $\wvec{\alpha}{\beta}$ and Proposition~\mref{pp:abb} also holds. Then the \gslwqsyms are specializations of bimonomial quasisymmetric functions corresponding to \ulwbs by taking $y_j$ to be $j, j\in \PP$, and the specialization map preserves products. Thus from Proposition~\mref{pp:abb} we obtain

\begin{coro} \mlabel{co:stqsh} The module $\GSym$ is a subalgebra of $\QQ[[X]]$ and the multiplication of two \gslwqsyms observes the quasi-shuffle relation:
$$ \widehat{M}_{\wvec{\alpha}{\beta}} \widehat{M}_{\wvec{\alpha'}{\beta'}} =
\widehat{M}_{\wvec{\alpha}{\beta}\ast \wvec{\alpha'}{\beta'}}.$$
\end{coro}

\section{Regularization of weak composition quasisymmetric functions}
\mlabel{sec:reg}

In this section, we temporally assume $\bfk =\RR$ in order to study the convergency of regularized quasisymmetric functions for weak compositions.

As noted in the introduction, our purpose of the paper is to make sense of the expressions
\begin{align}\mlabel{eq:malphamonomiabasis}
M_{\alpha}:=\sum_{I} x_I^\alpha
=\sum_{i_1<i_2<\cdots<i_k}x_{i_1}^{\alpha_1}x_{i_2}^{\alpha_2}\cdots x_{i_k}^{\alpha_k}
\end{align}
for all weak compositions $\alpha$. To begin with, we make a simple observation.

\begin{lemma} \mlabel{lem:wccong}
For any weak composition $\alpha$, the expression $M_\alpha$ in Eq.~\eqref{eq:malphamonomiabasis} is a well-defined formal power series in $\bfk[[X]]$ if and only if $\alpha$ is a left weak composition.
\end{lemma}

\begin{proof}
For a left weak composition $\alpha=(\alpha_1,\alpha_2,\cdots,\alpha_k)$, the {\bf left weak monomial quasisymmetric function} $M_\alpha$ is a well-defined power series and was studied in~\cite{YGZ17}.

If $\alpha$ is not left weak, then $\alpha_k=0$. Then we have
$$ M_\alpha=\sum_{i_1<\cdots< i_{k-1}}x_{i_1}^{\alpha_1} \cdots x_{i_{k-1}}^{\alpha_{k-1}} \Big(\sum_{i_k=i_{k-1}+1}^\infty x_{i_k}^{\alpha_k} \Big)
$$
which is not well defined since $\sum\limits_{i_k=i_{k-1}+1}^\infty x_{i_k}^{\alpha_k} =\sum\limits_{i_k=i_{k-1}+1}^\infty x_{i_k}^0$ is an infinite sum of $1$.
\end{proof}

To derive a well-defined formal power series from the expression in Eq.~\eqref{eq:malphamonomiabasis} when $\alpha$ is not left weak, we apply the Algebraic Birkhoff Factorization (Theorem~\mref{thm:Abdt}) in the algebraic approach of Connes and Kreimer~\cite{CK2000} to renormalization in perturbative quantum field theory. In practice we first give a regularization of such a formal expressions so that it makes sense and lands in a suitable Rota-Baxter algebra, and then take the renormalization by the Algebraic Birkhoff Factorization which will be carried out in Section~\mref{sec:renorm}.
Our regularization is of heat-kernel type as employed in~\cite{GPXZ,GZ08} for multiple zeta values.

For $\wvec{\alpha}{\beta}=\wvec{\alpha_1, \alpha_2, \cdots, \alpha_k}{\beta_1,\beta_2,\cdots,\beta_k}\in \DWC$ and $t\geq 0$, $\mre(z)<0$, we define
\begin{align}\mlabel{eq:mapphia}
\phi \bwvec{\alpha}{\beta}:=\phi \bwvec{\alpha}{\beta}(t):=\phi \bwvec{\alpha}{\beta}(t,z)
:=\sum_{i_1<i_2< \cdots<i_k}x_{i_1}^{\alpha_1}x_{i_2}^{\alpha_2}\cdots x_{i_k}^{\alpha_k}e^{(i_1+t)\beta_1 z}e^{(i_2+t)\beta_2 z}\cdots e^{(i_k+t)\beta_k z}.
\end{align}

\begin{theorem} \mlabel{thm:phi(lastpartlt0)}
The assignment in Eq.~\meqref{eq:mapphia} defines an algebra homomorphism
\begin{equation}\mlabel{eq:phi}
\phi: \bwcc\to \GSym[t][{z}^{-1}, {z}]],
\end{equation}
where $\GSym$ is the algebra of \gslwqsyms in Definition~\mref{de:stqsym}.
If $\wvec {\alpha}{\beta}$ is in $\DWC$ and $\alpha$ is a left weak composition, then
	$\phi\bwvec {\alpha}{\beta}$ is an element of $\GSym[t][[z]]$.
\end{theorem}

\begin{proof}
The proof is accomplished in the rest of this section, in two steps:
\smallskip

\noindent {\bf Step 1.}  We show that Eq.~\meqref{eq:mapphia} defines an algebra homomorphism
$$ \phi: \bwcc\to \QQ[[X]][t][z^{-1},z]].$$ This is proved in Proposition~\mref{pp:phialbeqtzx}.
\smallskip

\noindent
{\bf Step 2. }
We shows that the image of $\phi$ is in $\GSym[t][{z}^{-1}, {z}]]$ with the stated property when $\alpha$ is left weak. This is achieved in Proposition~\mref{pp:gsym}.
\end{proof}

\begin{defn}\mlabel{defn:regularizaion}
	We call the map $\phi$ in Eq.~\meqref{eq:phi} the \name{regularization map}, and  call $\phi\bwvec{\alpha}{\beta}$ the \name{regularization} of the formal expression $M_{\alpha}$ in Eq.~\meqref{eq:malphamonomiabasis} in the direction $\beta$.
\end{defn}

\begin{remark}
By Theorem~\mref{thm:phi(lastpartlt0)}, for each $\wvec {\alpha}{\beta}\in \DWC$, $\phi\bwvec {\alpha}{\beta}$ is an element in $\QQ [t][z^{-1},z]][[X]]$ and hence makes sense in $\bfk [t][z^{-1},z]][[X]]$ for any commutative ring $\bfk$ containing $\QQ$.
\end{remark}

Recall that near $z=0$,
\begin{align*}
\frac {z}{e^z-1}=\sum_{s=0}^{\infty}\frac{B_s}{s!}z^s,
\end{align*}
where $B_s \in \QQ $ is the $s$-th Bernoulli number. Thus for $z$ with real part $\mre (z)<0$, the series $\sum\limits_{s=1}^\infty e^{sz}$ is absolutely convergent and has a Laurent series expansion in $z^{-1}\QQ[[z]]\subset \QQ [z^{-1},z]]$:
\begin{align}\mlabel{eq:eiz1-eiz}
\sum_{s=1}^\infty e^{sz}=\frac {e^z}{1-e^z}=\sum_{s=0}^{\infty}\frac{b_s}{s!}z^{s-1},
\end{align}
where $b_0=-B_0=-1$, $b_1=-1-B_1=-\frac{1}{2}$ and $b_{s}=-B_{s}$ for $s\geq2$.

\begin {prop}
\begin{enumerate}
\item
Let $\wvec{\alpha}{\beta}=\wvec{\alpha_1, \alpha_2, \cdots, \alpha_k}{\beta_1,\beta_2,\cdots,\beta_k}$ be in $\DWC$. The expression $\phi\bwvec{\alpha}{\beta}$ in Eq.~\meqref{eq:mapphia} is a power series in $X$ whose coefficients have Laurent series expansion in $\QQ[t][z^{-1},z]]$. More precisely, 
taking $\bar\alpha$ to be the composition of positive entries of $\alpha$ and applying the abbreviations in Eq.~\meqref{eq:abb}, we have 
$\phi\bwvec{\alpha}{\beta}=\sum\limits_{I} c_{I} x_I^{\bar\alpha}$ in which
the coefficients $c_I$ are functions in $t$ and $z$ with Laurent expansions in $\QQ[t][z^{-1},z]]$:
\begin{equation}\mlabel{eq:phiimage2}
c_I:=\sum_{n\in \ZZ} c_{I,n} z^n \in \QQ[t][z^{-1},z]],
\quad
c_{I ,n}:=\sum_{s\in \NN} c_{I,n,s} t^s\in \QQ[t]
\end{equation}
that are convergent for $t\geq 0$ and $\mre(z)<0$. Thus
\begin{equation}\mlabel{eq:phiimage}
	\sum_{I} c_{I} x_I^{\bar\alpha}  =\sum_{I}\sum_{n\in \ZZ} c_{I ,n} z^n x_I^{\bar\alpha}
	 =\sum_{I}\sum_{n\in \ZZ}\sum_{s\in \NN} c_{I ,n,s} t^s z^n x_I^{\bar\alpha} \in \QQ [t][z^{-1},z]][[X]].
\end{equation}
\mlabel{it:phia1}
\item Let $\alpha_j>0$, $\alpha_{j+1}=\cdots=\alpha_k=0$.
As a Laurent series in $\QQ [t][z^{-1},z]]$ with variable $z$ and coefficients in $\QQ[t]$, the pole order of $c_{I}=\sum\limits_{I,n} c_{I,n}z^n$ in Item~\mref{it:phia1} is uniformly bounded by $k-j$, independent of $I$.  Furthermore, for a given $n\in\ZZ$, the degrees of the polynomials $\sum\limits_{s} c_{I,n,s} t^s$ in $t$ is uniformly bounded by $n+k-j$ independent of $I$.
\mlabel{it:phia2}
\item
The assignment in Eq.~\meqref{eq:mapphia} defines a map for which we still use the same letter
\begin{equation}
\phi: \bwcc \to \QQ[[X]][t][z^{-1},z]].
\mlabel{eq:phi2}
\end{equation}
\mlabel{it:phia3}
\item
The map $\phi$ is an algebra homomorphism.
\mlabel{it:phia4}
\end{enumerate}
\mlabel{pp:phialbeqtzx}
\end{prop}

The coefficients $c_{I,n,s}$ of $\phi\bwvec{\alpha}{\beta}$ in Eq.~\meqref{eq:phiimage} will be explicitly determined in Proposition~\mref{pp:gsym}.
Before the proof of Proposition~\mref{pp:phialbeqtzx}, we first illustrate by an example.

\begin{exam}\mlabel{exameq:phi0r}
	For any positive integer $r$, when $t\ge 0$ and $\mre (z)<0$, we have
	\begin{align*}
		\phi \bwvec {0}{r}=&\sum_{i=1}^\infty e^{(i+t)rz}=e^{trz}\sum_{i=1}^\infty e^{irz}=\sum_{n=0}^{\infty}\Big(\sum_{i=0}^n\frac{b_{n-i}}{i!(n-i)!}t^i\Big)(rz)^{n-1}\\
		=&-\frac {1}{rz}-\Big(t+\frac 12\Big)-\Big(\frac {t^2}2+\frac t2+\frac 1{12}\Big)rz-\Big(\frac{t^3}{6}+\frac{t^2}{4}+\frac{t}{12}\Big)(rz)^2+\cdots.\notag
	\end{align*}
\end{exam}

\begin {proof}
\mref{it:phia1} Let $\wvec{\alpha}{\beta}=\wvec{\alpha_1, \alpha_2, \cdots, \alpha_k}{\beta_1,\beta_2,\cdots,\beta_k}$.
If $\alpha =0^k$, then for $t\geq 0$ and $\mre (z)<0$, the series
$$\phi \bwvec{\alpha}{\beta}=\sum_{i_1<i_2< \cdots<i_k}e^{(i_1+t)\beta_1 z}e^{(i_2+t)\beta_2 z}\cdots e^{(i_k+t)\beta_k z}$$
is absolutely convergent and
$$\phi \bwvec{\alpha}{\beta}=e^{t(\beta_1+\beta_2+\cdots +\beta_k)z}\frac {e^{\beta _kz}}{1-e^{\beta _kz}}\frac {e^{(\beta _k+\beta _{k-1})z}}{1-e^{(\beta _k
		+\beta _{k-1})z}}\cdots \frac {e^{(\beta _k+\cdots +\beta _{1})z}}{1-e^{(\beta _k+\cdots +\beta _{1})z}}.
$$
The first factor in the product is in $\QQ [[tz]]$ which is a subalgebra of $\QQ [t][[z]]$. All the other factors  are in $\QQ [z^{-1},z]]$ by Eq.~\meqref{eq:eiz1-eiz}. Thus the product is in $\QQ [t][z^{-1},z]]$.

Now assume that $\alpha \not =0^k$ and let $\bar \alpha =(\alpha  _{j_1},\alpha  _{j_2}, \cdots, \alpha _{j_\ell})$, so $\ell =\ell (\bar \alpha )$. Then for fixed $i_{j_1}<i_{j_2}<\cdots <i_{j_{\ell }}$, the coefficient
$c_{i_{j_1},i_{j_2},\cdots,i_{j_{\ell }}}$ for the monomial
$x_{i_{j_1}}^{\alpha _{j_1}}x_{i_{j_2}}^{\alpha _{j_2}}\cdots x_{i_{j_\ell}}^{\alpha _{j_\ell}}$ of $\phi\bwvec{\alpha}{\beta}$ in Eq.~\meqref{eq:phiimage}
is
\begin{equation} \mlabel{eq:reg1}
	\begin{split}
		&\Big(\sum _{0<i_{1}<\cdots <i_{j_1}}e^{(i_{1}+t)\beta_1 z}\cdots e^{(i_{j_1}+t)\beta _{j_1}z}\Big)
		\Big(\sum _{i_{j_1}<i_{j_1+1}<\cdots<i_{j_2}}e^{(i_{j_1+1}+t)\beta_{j_1+1} z}\cdots  e^{(i_{j_2}+t)\beta _{j_2}z}\Big)\\
		&\cdots \Big(\sum _{i_{j_{\ell-1}}<i_{j_{\ell-1}+1}<\cdots<i_{j_\ell}}
		e^{(i_{j_{\ell-1}+1}+t)\beta_{j_{\ell-1}+1} z}\cdots  e^{(i_{j_\ell}+t)\beta _{j_\ell}z}\Big)
		\Big(\sum _{i_{j_\ell} <i_{j_\ell+1}<\cdots <i_{k}}
		e^{(i_{j_\ell+1}+t)\beta_{j_\ell+1} z}\cdots e^{(i_{k}+t)\beta _{k}z}\Big).
	\end{split}
\end{equation}
So if $\alpha $ is a left weak composition, then each factor is a finite sum, and thus is in $\QQ [t][[z]]$. If $\alpha $ is not left weak, then only the last sum
$$ \sum _{i_{j_\ell} <i_{j_\ell+1}<\cdots <i_{k}}
e^{(i_{j_\ell+1}+t)\beta_{j_\ell+1} z}\cdots e^{(i_{k}+t)\beta _{k}z}
$$
is infinite. But for $t\geq 0$ and $\mre (z)<0$, this sum is absolutely convergent to
\begin{equation} \mlabel{eq:phiend}
	e^{(i_{j_\ell}+t)(\beta _{j_\ell+1}+\cdots +\beta _k)z}\frac {e^{\beta _kz}}{1-e^{\beta _kz}}
	\frac {e^{(\beta _k+\beta _{k-1})z}}{1-e^{(\beta _k+\beta _{k-1})z}}\cdots \frac {e^{(\beta _k+\cdots +\beta _{j_\ell +1})z}}{1-e^{(\beta _k+\cdots +\beta _{j_\ell +1})z}}
\end{equation}
which is in $\QQ [t][z^{-1},z]]$, as in the case of $\alpha =0^k$. Since the other factors in the product in Eq.~\meqref{eq:reg1} are already in $\QQ [t][[z]]$, the product is in $\QQ [t][z^{-1},z]]$. Therefore $\phi\bwvec{\alpha}{\beta}$ is a power series with variables $X$ and coefficients $\QQ[t][z^{-1},z]]$.
\smallskip

\noindent

\mref{it:phia2} From the proof of Item~\mref{it:phia1}, the pole part of the Laurent series $c_I$ in $\QQ [t][z^{-1},z]]$ comes from the product in Eq.~\meqref{eq:phiend} (the case $\alpha =0^k$ corresponds to $ j=0$). Each fraction factor increases the pole order by one. Thus the pole order of $c_I$ is uniformly bounded by $k-j$ independent of $I$.

Furthermore, note that in Eqs.~(\mref{eq:reg1}) and (\mref{eq:phiend}), the variable $t$ only appears in the form $tz$ in the powers of  exponential functions in the product.
Thus as a Laurent series $c_I=\sum\limits_{n\in \ZZ} c_{I ,n} z^n\in \QQ [t][z^{-1},z]]$, the degree of $c_{I,n}=\sum\limits_{s} c_{I,n,s} t^s$ in $t$ of any power $z^n$ is uniformly bounded by $n+k-j$, independent of $I$.
\smallskip

\noindent
\mref{it:phia3}
An element in $\QQ [t][z^{-1},z]][[X]]$ is uniquely of the form $\sum\limits_{\gamma\in \C, I}c_{\gamma,I}x_I^\gamma$ with $I$ ordered $|\gamma|$-subsets of $\PP$ and  $c_{\gamma,I}$ in $\QQ[t][z^{-1},z]]$. 
Let $\QQ [t][^bz^{-1},z]][[X]]$ denote the subset of $\QQ [t][z^{-1},z]][[X]]$ consisting of the power series $\sum\limits_{\gamma\in \C,I} c_{\gamma, I}  x_I^{\gamma}$ for which the pole orders in $z$ of the coefficients $c_{\gamma,I}:=\sum\limits_{n\in \ZZ}c_{\gamma,I,n} z^n\in \QQ[t][z^{-1},z]]$ are uniformly bounded independent of $\gamma$ and $I$.
Then $\QQ [t][^bz^{-1},z]][[X]]$ is a subalgebra of $\QQ [t][z^{-1},z]][[X]]$.
Furthermore, the summation exchange 
$$ W: \QQ [t][^bz^{-1},z]][[X]] \to \QQ [t][[X]][z^{-1},z]],\quad  \sum_{\gamma\in\C,I} \sum_{n\in \ZZ} c_{\gamma\in \C,I ,n}z^n x_{I}^{\gamma}\mapsto \sum_{n\in \ZZ} \sum_{\gamma\in\C,I} c_{\gamma,I ,n}x_I^{\gamma} z^n$$
is well defined and is an algebra isomorphism.

Similarly, let $\QQ [^bt][[X]]$ denote the subset of $\QQ [t][[X]]$ consisting of all formal power series $\sum\limits_{\gamma\in \C,I} c_{\gamma,I} x_I^{\gamma}$ for which the degrees of the coefficients $c_{\gamma,I}: =\sum\limits_{s\in \NN} c_{\gamma,I,s} t^s\in\QQ[t]$ are uniformly bounded independent of $\gamma$ and $I$.
Then $\QQ [^bt][[X]]$ is a subalgebra of $\QQ [t][[X]]$.
Furthermore, the summation exchange 
$$ U: \QQ [^bt][[X]] \to \QQ [[X]][t], \quad \sum_{\gamma\in\C,I} \sum_{s\in \NN} c_{\gamma,I,s}t^s x_I^{\gamma} \mapsto  \sum_{s\in \NN} \sum_{\gamma\in\C,I} c_{\gamma,I,s}x_I^{\gamma} t^s,$$
is well defined and is an algebra isomorphism.

By Item~\mref{it:phia2}, the element
$\phi\bwvec{\alpha}{\beta}=\sum\limits_{I}\sum\limits_{n\in \ZZ}\sum\limits_{s\in \NN}c_{I,n,s} t^sz^nx_I^{\bar\alpha}$ is in $\QQ [t][^bz^{-1},z]][[X]]$ and hence corresponds to an element $\sum\limits_{n\in \ZZ}\sum\limits_{I}\sum\limits_{s\in \NN} c_{I,n,s} t^sx_I^{\bar\alpha} z^n$ in $\QQ[t][[X]][z^{-1},z]]$ by the isomorphism $W$.

Furthermore, for a given $n\in \ZZ$, the element $\sum\limits_{I}\sum\limits_{s\in \NN} c_{I,n,s} t^s x_I^{\bar\alpha}$ is in $\QQ[^bt][[X]]$ and hence corresponds to an element $\sum\limits_{s\in \NN}\sum\limits_{I} c_{I,n,s} x_I^{\bar\alpha} t^s$ in $\QQ[[X]][t]$.

In summary, the element $\phi\bwvec{\alpha}{\beta}$ defined in Eq.~\meqref{eq:mapphia} is an element of $\QQ[[X]][t][z^{-1},z]]$ after summation exchanges.

\smallskip

\noindent
\mref{it:phia4}
By Proposition~\mref{pp:abb}, the assignment $\wvec{\alpha}{\beta}\mapsto M_{\wvec{\alpha}{\beta}}(X,Y)$ defines an algebra homomorphism from $\bwcc$ to the algebra of  bimonomial quasisymmetric functions.

With the absolute convergence for $t\geq 0$ and $\mre (z)<0$ of the series $\phi\bwvec{\alpha}{\beta}$ in Eq.~\meqref{eq:mapphia} established in Item~\mref{it:phia1}, we see that the series is the specialization of $M_{\wvec{\alpha}{\beta}}(X,Y)$ by taking $y_{i_j}$ to be $e^{(i_j+t)z}$. The specialization map is an algebra homomorphism. Thus post-composing with the previous algebra homomorphism, we conclude that the assignment in Eq.~\meqref{eq:mapphia} is an algebra homomorphism. Since the maps $W$ and $U$ in the proof of Item~\mref{it:phia3} are algebra isomorphisms, we conclude that the map $\phi:\bwcc\to \QQ[[X]][t][z^{-1},z]]$ in Eq.~\meqref{eq:phi2} is an algebra homomorphism.
\end{proof}

We next carry out the second step in the proof of Theorem~\mref{thm:phi(lastpartlt0)}.
\begin{lemma}\mlabel{lem:phisor=sr*0}
Let $\wvec{\alpha}{\beta}=\jwvec{\alpha_1,\!\!\!\!&\cdots,\!\!\!\!&\alpha_j,\!\!\!\! &0,\!\!\!\!&\cdots,\!\!\!\!&0} {\beta_1,\!\!\!\!&\cdots,\!\!\!\!&\beta_j,\!\!\!\!&\beta_{j+1},\!\!\!\!&\cdots,\!\!\!\!&\beta_k}\in \DWC$ with $0\leq j\leq k$. Then for $t\ge 0$ and $\mre (z)<0$, 
\begin{align*}
\phi \bwvec{\alpha}{\beta}(t) =\phi \bjwvec{\alpha_1,\!\!\!\!&\cdots,\!\!\!\!&\alpha_{j-1},\!\!\!\!&\alpha_j}
{\beta_{1},\!\!\!\!&\cdots,\!\!\!\!&\beta_{j-1},\!\!\!\!&\beta_j+\beta_{j+1}+\cdots+\beta_k}(t)\;
\phi \bjwvec{0,\!\!\!\!&\cdots,\!\!\!\!&0}{\beta_{j+1},\!\!\!\!&\cdots,\!\!\!\!&\beta_k}(0).
\end{align*}
\end{lemma}
\begin{proof}
Since $\alpha=(\alpha_1,\cdots,\alpha_j,0,\cdots,0)$, by Eq. \meqref{eq:mapphia} we have
\begin{align*}
\phi\bwvec{\alpha}{\beta}(t)
=&\sum_{i_1<\cdots<i_j}x_{i_1}^{\alpha_1}\cdots x_{i_j}^{\alpha_j}e^{(i_1+t)\beta_1 z}\cdots e^{(i_j+t)\beta_j z}\sum_{i_j<i_{j+1}<\cdots<i_k}e^{(i_{j+1}+t)\beta_{j+1}z}\cdots e^{(i_k+t)\beta_k z}\\
=&\sum_{i_1<\cdots<i_j}x_{i_1}^{\alpha_1}\cdots x_{i_{j-1}}^{\alpha_{j-1}}x_{i_j}^{\alpha_j}e^{(i_1+t)\beta_1z}\cdots e^{(i_{j-1}+t)\beta_{j-1}z} e^{(i_j+t)(\beta_j+\beta_{j+1}+\cdots+\beta_{k})z}\sum_{0<i_{j+1}<\cdots<i_k}e^{i_{j+1}\beta_{j+1}z}\cdots e^{i_k\beta_k z}\\
=&\,\phi  \bjwvec{\alpha_1,\!\!\!\!&\cdots,\!\!\!\!&\alpha_{j-1},\!\!\!\!&\alpha_j}{\beta_1,\!\!\!\!&\cdots,\!\!\!\!&\beta_{j-1},\!\!\!\!&\beta_j+\beta_{j+1}+\cdots+\beta_k}(t)\;
\phi \bjwvec{0,\!\!\!\!&\cdots,\!\!\!\!&0}{\beta_{j+1},\!\!\!\!&\cdots,\!\!\!\!&\beta_k}(0),
\end{align*}
as required.
\end{proof}

\begin{coro}\mlabel{cor:phisor=sr0}
For any composition ${\beta}=(\beta_1,\beta_2,\cdots,\beta_k)$, when $t\ge 0$ and $\mre (z)<0$, the expression $\phi\bwvec{0^k}{\beta}$ is equal to
\begin{align*}
\sum_{n=0}^{\infty} \left(\sum_{(s_1,s_2,\cdots,s_k)\in WC(n)}\Big(\sum_{i=0}^{s_1}\frac{b_{s_1-i}}{i!(s_1-i)!}t^i\Big)\frac{b_{s_2}\cdots b_{s_k}}{s_2!\cdots s_k!}
{(\beta_1+\cdots+\beta_{k})}^{s_1-1}{(\beta_2+\cdots+\beta_{k})}^{s_2-1}\cdots \beta_k^{s_k-1}\right)z^{n-k}.
\end{align*}
\end{coro}
\begin{proof}
Applying Lemma~\mref{lem:phisor=sr*0} repeatedly, we obtain
\begin{align*}
\phi\bwvec{0^k}{\beta} =&\phi\bwvec{0}{\beta_1+\cdots+\beta_k}(t)\phi\bwvec{0}{\beta_2+\cdots+\beta_k}(0) \cdots\phi\bwvec{0}{\beta_k}(0).
\end{align*}
Taking $t=0$ in Example~\mref{exameq:phi0r}, we obtain
\begin{align*}
\phi\bwvec{0}{r}(0)=\sum_{n=0}^{\infty}\frac{b_n}{n!}(rz)^{n-1},
\end{align*}
and hence
\begin{align*}
\phi\bwvec{0^k}{\beta}
=\Big(\sum_{n=0}^{\infty}\Big(\sum_{i=0}^n\frac{b_{n-i}}{i!(n-i)!}t^i\Big)(\beta_1+\cdots+\beta_k)^{n-1}z^{n-1}\Big)
\Big(\prod_{j=2}^k\sum_{i=0}^{\infty}\frac{b_i}{i!}{(\beta_j+\cdots+\beta_{k})}^{i-1}z^{i-1}\Big),
\end{align*}
yielding the desired result.
\end{proof}

Now we can narrow down the range of the algebra homomorphism $\phi$ and complete the proof of Theorem~\mref{thm:phi(lastpartlt0)}.

\begin{prop}\mlabel{pp:gsym}
Let $\wvec {\alpha}{\beta}=\wvec{\alpha_1,\alpha_2,\cdots,\alpha_k}{\beta_1,\beta_2,\cdots,\beta_k}\in\DWC$.
Let $j$ with $0\leq j\leq k$ such that
$\alpha_j>0$ and $\alpha_{j+1}=\cdots=\alpha_k=0$. Then for $t\ge 0$ and $\mre (z)<0$, we have
$$
\phi\bwvec{\alpha}{\beta} =\sum_{n=0}^{\infty}c_{n,\wvec{\alpha}{\beta}}z^{n-k+j},
$$
with coefficients $c_{n,\wvec{\alpha}{\beta}}$ in the algebra $\GSym[t]$ of \gslwqsyms defined in Definition~\mref{de:stqsym}.
	If $\wvec {\alpha}{\beta}\in \DWC$ and $\alpha$ is a left weak composition, then
	$\phi\bwvec {\alpha}{\beta}$ is an element of $\GSym[t][[z]]$.
In fact,
\begin{equation}
\begin{array}{l}
c_{n,\wvec{\alpha}{\beta}}=\sum\limits_{i=0}^n\left(\sum\limits_{(s_1,\cdots,s_k)\in WC(n-i)}\frac{b_{s_{j+1}}\cdots b_{s_k}}{s_1!\cdots s_k!}\beta_1^{s_1}\cdots \beta_{j-1}^{s_{j-1}}(\beta_j+\cdots+\beta_k)^{s_j}(\beta_{j+1}+\cdots+\beta_{k})^{s_{j+1}-1}\right.\\
\qquad \qquad \qquad \left . (\beta_{j+2}+\cdots+\beta_{k})^{s_{j+2}-1}\cdots \beta_k^{s_k-1}\widehat{M}_{\wvec{\alpha_1,\alpha_2,\cdots,\alpha_j}{s_1,s_2,\cdots,s_j}}
\right)\frac{(\beta_1+\cdots+ \beta_k)^{i}}{i!}t^{i}.
\end{array}
\mlabel{eq:coef}
\end{equation}
\end{prop}

\begin{proof}
We just need to prove Eq.~\meqref{eq:coef} since it shows that $c_{n,\wvec{\alpha}{\beta}}$ is a polynomial in $t$ with coefficients in $\GSym$.
When $\alpha_k>0$, $t\ge 0$ and $\mre (z)<0$, we have
\begin{align*}
\phi \bwvec{\alpha}{\beta}
&=\sum_{i_1<\cdots<i_k}x_{i_1}^{\alpha_1}\cdots x_{i_k}^{\alpha_k}e^{(i_1+t)\beta_1z} \cdots e^{(i_k+t)\beta_kz}\\
&=e^{t(\beta_1+\cdots \beta_k)z}\sum_{i_1<\cdots<i_k}x_{i_1}^{\alpha_1}\cdots x_{i_k}^{\alpha_k}e^{(i_1\beta_1+\cdots i_k\beta_k)z}\\
&=\Big(\sum_{n=0}^{\infty}\frac{t^n(\beta_1+\cdots \beta_k)^nz^n}{n!}\Big)\Big(\sum_{i_1<\cdots<i_k}x_{i_1}^{\alpha_1}\cdots x_{i_k}^{\alpha_k}\sum_{n=0}^{\infty}\frac{(i_1\beta_1+ \cdots i_k\beta_k)^n}{n!}z^n\Big)\\
&=\Big(\sum_{n=0}^{\infty}\frac{t^n(\beta_1+\cdots \beta_k)^nz^n}{n!}\Big)\Big(\sum_{n=0}^{\infty}z^n\sum_{(s_1,\cdots,s_k)\in WC(n)}\frac{\beta_1^{s_1}\cdots \beta_k^{s_k}}{s_1!\cdots s_k!}\sum_{i_1<\cdots<i_k}i_1^{s_1}\cdots i_k^{s_k}x_{i_1}^{\alpha_1}\cdots x_{i_k}^{\alpha_k}\Big)\\
&=\sum_{n=0}^{\infty}\Big(\sum_{(s_0,s_1,\cdots,s_k)\in WC(n)}\frac{(\beta_1+\cdots \beta_k)^{s_0}\beta_1^{s_1}\cdots \beta_k^{s_k}}{s_0!s_1!\cdots s_k!}t^{s_0}\widehat{M}_{\wvec{\alpha_1,\cdots,\alpha_k}{s_1,\cdots,s_k}}\Big)z^n.
\end{align*}
Then, by Lemma \mref{lem:phisor=sr*0} and Corollary \mref{cor:phisor=sr0},
\begin{align*}
c_{n,\wvec{\alpha}{\beta}}=\sum_{(s_0,s_1,\cdots,s_k)\in WC(n)}&\frac{t^{s_0}b_{s_{j+1}}\cdots b_{s_k}}{s_0!s_1!\cdots s_k!}(\beta_1+\cdots+ \beta_k)^{s_0}\beta_1^{s_1}\cdots \beta_{j-1}^{s_{j-1}}(\beta_j+\cdots+\beta_k)^{s_j}\\
&(\beta_{j+1}+\cdots+\beta_{k})^{s_{j+1}-1}(\beta_{j+2}+\cdots+\beta_{k})^{s_{j+2}-1}\cdots \beta_k^{s_k-1}\widehat{M}_{\wvec{\alpha_1,\alpha_2,\cdots,\alpha_j}{s_1,s_2,\cdots,s_j}}.
\end{align*}
This gives the desired form of $c_{n,\wvec{\alpha}{\beta}}$.
\end{proof}

\begin {exam}\mlabel{examphi0r0ss0} For any $s, r,r_1,r_2\in\PP$, when $t\ge 0$ and $\mre (z)<0$, we have
\begin{align*}
\phi\bwvec {s}{r} =\sum_{n=0}^{\infty}\Big(\sum_{i=0}^n\frac{ t^{n-i}}{i!(n-i)!}\widehat{M}_{\wvec{s}{i}}\Big)(rz)^n,
\end{align*}
\begin{align*}
\phi \bwvec {s,0}{r_1,r_2}=\sum_{n=0}^{\infty}\sum_{(s_0,s_1,s_2)\in WC(n)}\frac{t^{s_0}b_{s_2}}{s_0!s_1!s_2!}(r_1+r_2)^{s_0+s_1}r_2^{s_2-1}\widehat{M}_{\wvec{s}{s_1}}z^{n-1}
\end{align*}
and
\begin{align*}
\phi  \bwvec {0,s}{r_1,r_2}
=&\sum_{n=0}^{\infty}\Big(\sum_{(s_0,s_1,s_2)\in WC(n)}\frac{(r_1+r_2)^{s_0}r_1^{s_1}r_2^{s_2}}{s_0!s_1!s_2!}t^{s_0}\widehat{M}_{\wvec{0,s}{s_1,s_2}}\Big)z^n.
\end{align*}
\end{exam}

\section{Renormalization of quasisymmetric functions}
\mlabel{sec:renorm}
We now apply the Algebraic Birkhoff Factorization to derive the renormalized quasisymmetric functions for weak compositions.

\subsection {Directional quasisymmetric functions}

Applying the Algebraic Birkhoff Factorization in Theorem~\mref{thm:Abdt} to the algebra homomorphism
$$\phi: \bwcc\to \GSym[t][{z}^{-1},{z} ]]\subseteq \bfk[t][[X]][{z} ^{-1}, {z} ]] $$
from Theorem~\mref{thm:phi(lastpartlt0)}, we obtain unique algebra homomorphisms
$$\phi_{-}: \bwcc \to  \GSym[t][{z}^{-1}]\quad \text{ and } \quad \phi_{+}: \bwcc \to  \GSym[t][[{z} ]],$$
defined by Eqs \meqref{eq:phi-xabd} and \meqref{eq:phi+xabd} respectively, such that
$\phi=\phi_{+}\star\phi_{-}^{\star\, (-1)}.$
Here the Rota-Baxter operator $P$ on $\bfk[t][[X]][{z} ^{-1}, {z} ]]$ is given by Eq.~\meqref{eq:defRBO1}.
More precisely, for any $\wvec{\alpha}{\beta}=\wvec{\alpha_1,\alpha_2,\cdots,\alpha_k}{\beta_1,\beta_2,\cdots,\beta_k}\in\DWC$, we have
\begin{align*}
\phi_{-}\bwvec{\alpha}{\beta} =-P \Big(\phi\bwvec{\alpha}{\beta}+
\sum_{i=1}^{k-1}\phi\bwvec{\alpha_1,\cdots,\alpha_i}{\beta_{1},\cdots,\beta_i} \phi_{-}\bwvec{\alpha_{i+1},\cdots,\alpha_k}{\beta_{i+1},\cdots,\beta_k} \Big )
\end{align*}
and
\begin{align*}
\phi_{+}\bwvec{\alpha}{\beta} =(\id-P) \Big(\phi\bwvec{\alpha}{\beta}+
\sum_{i=1}^{k-1}\phi\bwvec{\alpha_1,\cdots,\alpha_i}{\beta_{1},\cdots,\beta_i} \phi_{-}\bwvec{\alpha_{i+1},\cdots,\alpha_k}{\beta_{i+1},\cdots,\beta_k}  \Big).
\end{align*}

\begin {defn}\mlabel{defn:zetaqsf}
For any $\wvec{\alpha}{\beta}=\wvec{{\alpha_1,\alpha_{2},\cdots, \alpha_k}}{{\beta_1,\beta_{2},\cdots,\beta_k}}\in \DWC$, the evaluation
$$Z\bwvec{\alpha}{\beta}:=\phi_{+}\bwvec{\alpha}{\beta}\Big|_{{z} =0}\in \GSym[t]$$
is called the {\bf directional quasisymmetric function} of $\alpha$ in direction $\beta$.
\end{defn}

A property of directional quasisymmetric functions is apparent:
\begin{coro}\mlabel{cor:Z(s,s)starrul}
The directional  quasisymmetric functions satisfy the quasi-shuffle relation, that is, for any $\wvec{\alpha}{\beta}$ and $\wvec{\alpha'}{\beta'}$ in $\DWC$, we have
\vspace{-.3cm}
\begin{align*}
Z\Big(\wvec{\alpha}{\beta}\Big)Z\Big(\wvec{\alpha'}{\beta'}\Big)=Z\Big(\wvec{\alpha}{\beta}*\wvec{\alpha'}{\beta'}\Big).
\end{align*}
\end{coro}
\begin{proof}
By Theorem \mref{thm:Abdt}, $\phi_{+}: \bwcc \to  \GSym[t][[{z} ]]$ is an algebra homomorphism. The evaluation at $z=0$ is also an algebra homomorphism. Thus their composition $Z$ is still one, giving the equation.
\end{proof}

\begin{defn}
Let $k$ be a positive integer and $I=(i_1,i_2,\cdots,i_p)$ a composition of $k$. For $1\leq j\leq p$, define the partial sum $I_j=i_1+\cdots+i_j$ with the convention that $I_0=0$.
The {\bf partition vectors} of a weak composition $\alpha=(\alpha_1,\alpha_2,\cdots,\alpha_k)$ from the composition $I=(i_1,i_2,\cdots,i_p)$ are the vectors
$\alpha^{ (j)}=(\alpha_{I_{j-1}+1},\cdots,\alpha_{I_j})$, $1\leq j\leq p$.
\end{defn}

With this notion, we obtain the following analog of \cite[Theorem 3.8]{GZ08} by the same argument.

\begin{lemma}\mlabel{thm:phi+-}
Let $P:\GSym[t][{z}^{-1},{z}]]\rightarrow \GSym[t][{z}^{-1},{z}]]$ be the Rota-Baxter operator in Eq.~\meqref{eq:defRBO1}. Denote $\check{P}=-P$ and $\tilde{P}=\id-P$.
For any $\wvec{\alpha}{\beta}=\wvec{\alpha_1,\alpha_2,\cdots,\alpha_k}{\beta_1,\beta_2,\cdots,\beta_k}\in\DWC$, we have
\begin{align*}
\phi_{-}\bwvec{\alpha}{\beta}
=\sum_{(i_1,\cdots,i_p)\in C(k)}
\check{P}\Big(\phi\bwvec{\alpha^{(1)}}{\beta^{(1)}} \cdots\check{P}\Big(\phi\bwvec{\alpha^{(p-1)}}{\beta^{(p-1)}}
\check{P}\Big(\phi\bwvec{\alpha^{(p)}}{\beta^{(p)}}\Big)\Big)\cdots\Big)
\end{align*}
and
\begin{align*}
\phi_{+}\bwvec{\alpha}{\beta}
=&\sum_{(i_1,\cdots,i_p)\in{C(k)}}
\tilde{P}\Big(\phi\bwvec{\alpha^{\,\,(1)}}{\beta^{\,\,(1)}}\cdots
\check{P}\Big(\phi\bwvec{\alpha^{(p-1)}}{\beta^{(p-1)}}\check{P}\Big(\phi\bwvec{\alpha^{(p)}}{\beta^{(p)}}\Big)\Big)\cdots\Big).
\end{align*}
\end{lemma}

The following result shows that for \ulwbs, the directional quasisymmetric functions agree with quasisymmetric functions, as expected.

\begin{coro}\mlabel{cor:phi+=phiskneq0}
Let $\wvec{\alpha}{\beta}\in \DWC$ be an \ulwb.
Then $Z\bwvec{\alpha}{\beta}=M_{{\alpha}}$, independent of the choice of the composition $\beta$.
\end{coro}
\begin{proof}
Let $\wvec{\alpha}{\beta}=\wvec{{\alpha_1,\alpha_{2},\cdots, \alpha_k}}{{\beta_1,\beta_{2},\cdots,\beta_k}}$.
If a composition $(i_1,i_2,\cdots,i_p)$ of ${k}$ is not $(k)$, then $p>1$. The last component of the partition vector $\alpha^{(p)}$ is $\alpha_k$ which is a positive integer. Then, by Theorem \mref{thm:phi(lastpartlt0)},
$\phi\bwvec{\alpha^{(p)}}{\beta^{(p)}}$ is in  $\GSym[t][[{z}]]$ so that
$\check{P}\Big(\phi\bwvec{\alpha^{(p)}}{\beta^{(p)}}\Big)=0$.
Hence in the sum of $\phi_+\bwvec{\alpha}{\beta}$ in Lemma \mref{thm:phi+-}, only the term with $(i_1,\cdots,i_p)=(k)$ is left. So
$$\phi_{+}\bwvec{\alpha}{\beta} =\tilde{P}\Big(\phi\bwvec{\alpha}{\beta}\Big)=\phi\bwvec{\alpha}{\beta}.$$
It then follows from Proposition~\mref{pp:gsym} that $Z\bwvec{\alpha}{\beta}=\phi\bwvec{\alpha}{\beta}\Big|_{{z}=0}
=\widehat{M}_{\wvec{\alpha}{0^k}}=M_{{\alpha}}$.
\end{proof}

We also give examples when $\alpha$ are weak compositions.

\begin {exam}\mlabel{example:z(sr)}
Continuing with the examples in Example \mref{examphi0r0ss0}, we have
\begin{align*}Z \bwvec {0}{r}=-t-\frac 12,\quad
Z\bwvec {s}{r}=\widehat{M}_{\wvec{s}{0}}=M_{{s}}\quad{\rm and }\quad
Z\bwvec {0,s}{r_1,r_2}=\widehat{M}_{\wvec{0,s}{0,0}}=M_{(0,s)}.
\end{align*}
From  Lemma \mref{thm:phi+-} it follows that
\begin{align*}
\phi_{+}\Big(\wvec{{s,0}}{{r_1,r_2}}\Big)=\tilde{P}\Big(\phi\Big(\wvec{{s,0}}{{r_1,r_2}}\Big) \Big)
+\tilde{P}\Big(\phi\Big(\wvec{{s}}{{r_1}}\Big)\check{P}\Big(\phi\bwvec{{0}}{{r_2}} \Big)\Big).
\end{align*}
Hence by Example \mref{examphi0r0ss0}, we have
\begin{align*}
Z\bwvec {s,0}{\beta_1,\beta_2}=\phi_{+}\bwvec{{0,s}}{{\beta_2,\beta_1}}\Big|_{z=0}
= -(t+\frac{1}{2})\widehat{M}_{\wvec{s}{0}}-\widehat{M}_{\wvec{s}{1}} = -(t+\frac{3}{2})M_s-M_{(0,s)}. 
\end{align*}
\end{exam}

\subsection {Independence of the direction vector}
We now extract a renormalized quasisymmetric function for a weak composition that is independent of the directional vector.

For any $\wvec{\alpha}{\beta}=\wvec{{\alpha_1,\alpha_{2},\cdots, \alpha_k}}{{\beta_1,\beta_{2},\cdots,\beta_k}}\in \DWC$,
there exists a unique nonnegative integer $j\leq k$ such that
$\alpha_j>0$ and $\alpha_{j+1}=\cdots=\alpha_{k}=0$.
Let $\mathfrak{S}_{\alpha}$ be the permutation group $\mathbb{S}_{k-j}$ of $[k-j]$.
Define
$$ Z\bwvec{\alpha}{\beta}^{\mathfrak{S}_{\alpha}}:=\sum_{\sigma\in \mathfrak{S}_{\alpha}}
Z\bjwvec{\alpha_1,\!\!\!\!&\cdots,\!\!\!\!&\alpha_{j},\!\!\!\!&\alpha_{j+1},\!\!\!\!&\cdots,\!\!\!\!&\alpha_{k}}
{\beta_1,\!\!\!\!&\cdots,\!\!\!\!&\beta_{j},\!\!\!\!&\beta_{j+\sigma (1)},\!\!\!\!&\cdots,\!\!\!\!&\beta_{j+\sigma (k-j)})}.$$

\begin{lemma}
For any $\wvec{\alpha}{\beta}\in \DWC$,
the element $Z\bwvec{\alpha}{\beta}^{\mathfrak{S}_{\alpha}}$ is independent of the choice of the composition $\beta$.
\mlabel{thm:gzeta}
\end{lemma}
\begin{proof}
Let $j$  be the nonnegative integer such that
$\alpha_j>0$ and $\alpha_{j+1}=\cdots=\alpha_{k}=0$.
We prove by induction on $k-j\geq 0$. By Corollary \mref{cor:phi+=phiskneq0}, the assertion is
true for $k-j=0$.
Suppose that the statement has been proved for $k-j=m$ for a given $m\geq 0$, and consider $k-j=m+1$.
Then $\alpha_k=0$.
Denote $\alpha'=(\alpha_1,\cdots, \alpha_{k-1})$ and $\beta'=(\beta_1,\cdots, \beta_{k-1})$.
Then by Corollary~\mref{cor:Z(s,s)starrul}, we have
\begin{align*}
Z\bwvec{\alpha'}{\beta'}^{\mathfrak{S}_{\alpha'}}*Z\bwvec{{0}}{{\beta_{k}}}
=&\sum_{\sigma\in\mathfrak{S}_{\alpha'}}Z\bjwvec{\alpha_1,\!\!\!\!&\cdots,\!\!\!\!&\alpha_j,\!\!\!\!&\alpha_{j+1},\!\!\!\!&\cdots,\!\!\!\!&\alpha_{k-1}}
{\beta_1,\!\!\!\!&\cdots,\!\!\!\!&\beta_{j},\!\!\!\!&\beta_{j+\sigma(1)},\!\!\!\!&
\cdots,\!\!\!\!&\beta_{j+\sigma(k-j-1)}}*Z\bwvec{{0}}{{\beta_{k}}}\\
=&\sum _{\sigma\in\mathfrak{S}_{\alpha'}}\sum _{l=1}^jZ\bjwvec{\alpha_1,\!\!\!\!&\cdots,\!\!\!\!&\alpha_{l-1},\!\!\!\!&0,\!\!\!\!&\alpha_l,\!\!\!\!&\cdots,\!\!\!\!&\alpha_j,\!\!\!\!&\alpha_{j+1},
\!\!\!\!&\cdots,\!\!\!\!&\alpha_{k-1}}
{\beta_1,\!\!\!\!&\cdots,\!\!\!\!&\beta_{l-1},\!\!\!\!&\beta_{k},\!\!\!\!&\beta_{l},\!\!\!\!&\cdots,\!\!\!\!&\beta_{j},\!\!\!\!&\beta_{j+\sigma(1)},\!\!\!\!&\cdots,\!\!\!\!&\beta_{j+\sigma(k-1-j)}}\\
&+\sum _{\sigma\in\mathfrak{S}_{\alpha'}}\sum _{l=1}^jZ\bjwvec{\alpha_1,\!\!\!\!&\cdots,\!\!\!\!&\alpha_{l-1},\!\!\!\!&\alpha_l,\!\!\!\!&\cdots,\!\!\!\!&\alpha_j,\!\!\!\!&\alpha_{j+1},\!\!\!\!&\cdots,\!\!\!\!&\alpha_{k-1}}
{\beta_1,\!\!\!\!&\cdots,\!\!\!\!&\beta_{l-1},\!\!\!\!&\beta_{l}+\beta_{k},\!\!\!\!&\cdots,\!\!\!\!&\beta_{j},\!\!\!\!&\beta_{j+\sigma(1)},\!\!\!\!&\cdots,\!\!\!\!&\beta_{j+\sigma(k-1-j)}}\\
&+\sum _{\sigma\in\mathfrak{S}_{\alpha'}}\sum _{l=1}^{k-j}Z\bjwvec{\alpha_1,\!\!\!\!&\cdots,\!\!\!\!&\alpha_j,\!\!\!\!&\alpha_{j+1},\!\!\!\!&\cdots,\!\!\!\!&\alpha_{j+l-1},\!\!\!\!&0,\!\!\!\!&\alpha_{j+l},\!\!\!\!&\cdots,\!\!\!\!&\alpha_{k-1}}
{\beta_1,\!\!\!\!&\cdots,\!\!\!\!&\beta_{j},\!\!\!\!&\beta_{j+\sigma(1)},\!\!\!\!&\cdots,\!\!\!\!&\beta_{j+\sigma(l-1)},\!\!\!\!&\beta_{k},\!\!\!\!&\beta_{j+\sigma(l)},\!\!\!\!&\cdots,\!\!\!\!&\beta_{j+\sigma(k-1-j)}}\\
&+\sum _{\sigma\in\mathfrak{S}_{\alpha'}}\sum _{l=1}^{k-j-1}Z\bjwvec{\alpha_1,\!\!\!\!&\cdots,\!\!\!\!&\alpha_j,\!\!\!\!&\alpha_{j+1},\!\!\!\!&\cdots,\!\!\!\!&\alpha_{j+l-1},\!\!\!\!&\alpha_{j+l},\!\!\!\!&\cdots,\!\!\!\!&\alpha_{k-1}}
{\beta_1,\!\!\!\!&\cdots,\!\!\!\!&\beta_{j},\!\!\!\!&\beta_{j+\sigma(1)},\!\!\!\!&\cdots,\!\!\!\!&\beta_{j+\sigma(l-1)},\!\!\!\!&\beta_{j+\sigma(l)}+\beta_{k},\!\!\!\!&\cdots,\!\!\!\!&\beta_{j+\sigma(k-1-j)}}
\end{align*}
By Example \mref{example:z(sr)}, $Z\bwvec{{0}}{{\beta_{k}}}=-t-\frac{1}{2}$.
Applying the induction hypothesis shows that the left-hand side and the first two terms on the right-hand side are independent of $\beta$.
For the fourth term on the right-hand side, we have 
\begin{align*}&\sum _{\sigma\in\mathfrak{S}_{\alpha'}}\sum _{l=1}^{k-j-1}Z\bjwvec{\alpha_1,\!\!\!\!&\cdots,\!\!\!\!&\alpha_j,\!\!\!\!&\alpha_{j+1},\!\!\!\!&\cdots,\!\!\!\!&\alpha_{j+l-1},\!\!\!\!&\alpha_{j+l},\!\!\!\!&\cdots,\!\!\!\!&\alpha_{k-1}}
{\beta_1,\!\!\!\!&\cdots,\!\!\!\!&\beta_{j},\!\!\!\!&\beta_{j+\sigma(1)},\!\!\!\!&\cdots,\!\!\!\!&\beta_{j+\sigma(l-1)},\!\!\!\!&\beta_{j+\sigma(l)}+\beta_{k},\!\!\!\!&\cdots,\!\!\!\!&\beta_{j+\sigma(k-1-j)}}\\
=&\sum _{l=1}^{k-j-1}\sum _{\sigma\in\mathfrak{S}_{\alpha'}}Z\bjwvec{\alpha_1,\!\!\!\!&\cdots,\!\!\!\!&\alpha_j,\!\!\!\!&\alpha_{j+1},\!\!\!\!&\cdots,\!\!\!\!&\alpha_{j+l-1},\!\!\!\!&\alpha_{j+l},\!\!\!\!&\cdots,\!\!\!\!&\alpha_{k-1}}
{\beta_1,\!\!\!\!&\cdots,\!\!\!\!&\beta_{j},\!\!\!\!&\beta_{j+\sigma(1)},\!\!\!\!&\cdots,\!\!\!\!&\beta_{j+\sigma(l-1)},\!\!\!\!&\beta_{j+\sigma(l)}+\beta_{k},\!\!\!\!&\cdots,\!\!\!\!&\beta_{j+\sigma(k-1-j)}}\\
=&\sum _{l=1}^{k-j-1}\sum _{s=1}^{k-j-1}\sum _{\sigma\in\mathfrak{S}_{\alpha'}, \sigma (l)=s}Z\bjwvec{\alpha_1,\!\!\!\!&\cdots,\!\!\!\!&\alpha_j,\!\!\!\!&\alpha_{j+1},\!\!\!\!&\cdots,\!\!\!\!&\alpha_{j+l-1},\!\!\!\!&\alpha_{j+l},\!\!\!\!&\cdots,\!\!\!\!&\alpha_{k-1}}
{\beta_1,\!\!\!\!&\cdots,\!\!\!\!&\beta_{j},\!\!\!\!&\beta_{j+\sigma(1)},\!\!\!\!&\cdots,\!\!\!\!&\beta_{j+\sigma(l-1)},\!\!\!\!&\beta_{j+s}+\beta_{k},\!\!\!\!&\cdots,\!\!\!\!&\beta_{j+\sigma(k-1-j)}}\\
=&\sum _{s=1}^{k-j-1}\sum _{l=1}^{k-j-1}\sum _{\sigma\in\mathfrak{S}_{\alpha'} ,\sigma (l)=s}Z\bjwvec{\alpha_1,\!\!\!\!&\cdots,\!\!\!\!&\alpha_j,\!\!\!\!&\alpha_{j+1},\!\!\!\!&\cdots,\!\!\!\!&\alpha_{j+l-1},\!\!\!\!&\alpha_{j+l},\!\!\!\!&\cdots,\!\!\!\!&\alpha_{k-1}}
{\beta_1,\!\!\!\!&\cdots,\!\!\!\!&\beta_{j},\!\!\!\!&\beta_{j+\sigma(1)},\!\!\!\!&\cdots,\!\!\!\!&\beta_{j+\sigma(l-1)},\!\!\!\!&\beta_{j+s}+\beta_{k},\!\!\!\!&\cdots,\!\!\!\!&\beta_{j+\sigma(k-1-j)}}\\
=&\sum _{s=1}^{k-j-1}Z\bjwvec{\alpha_1,\!\!\!\!&\cdots,\!\!\!\!&\alpha_j,\!\!\!\!&\alpha_{j+1},\!\!\!\!&\cdots,\!\!\!\!&\alpha_{j+l-1},\!\!\!\!&\alpha_{j+l},\!\!\!\!&\cdots,\!\!\!\!&\alpha_{k-1}}
{\beta_1,\!\!\!\!&\cdots,\!\!\!\!&\beta_{j},\!\!\!\!&\beta_{j+1},\!\!\!\!&\cdots,\!\!\!\!&\beta_{j+s-1},\!\!\!\!&\beta_{j+s}+\beta_{k},\!\!\!\!&\cdots,\!\!\!\!&\beta_{k-1}}^{\mathfrak{S}_{\alpha'}},
\end{align*}
which is also independent of $\beta$ by the induction hypothesis.
So the third term on the right-hand side is independent of $\beta$.
Note that $\alpha_j>0$ and $\alpha_{j+1}=\cdots=\alpha_{k}=0$, we see the  third term  on the right-hand side equals $Z\bwvec{\alpha}{\beta}^{\mathfrak{S}_{\alpha}}$. This completes the induction.
\end{proof}

The next result shows that for any positive integer $\delta$, $Z\bwvec{\alpha}{\alpha+\delta}$ is independent of the choice of $\delta$, where
$\alpha+\delta:=(\alpha_1+\delta,\cdots,\alpha_k+\delta)$.

\begin{theorem}\mlabel{thm:mainthmwildms}
Let $\alpha=(\alpha_1,\alpha_2,\cdots,\alpha_k)$ be a weak composition. Then for any $\delta\in \mathbb{P}$,
\begin{enumerate}
\item\mlabel{thmitem:skgeq0} when $\alpha_k>0$, we have $Z\bwvec{\alpha}{\alpha+\delta}=M_{{\alpha}}$;
\item\mlabel{thmitem:sk=0} when $\alpha_k=0$, we have
\begin{align*}
Z\bwvec{\alpha}{\alpha+\delta}=\frac{1}{|{\mathfrak{S}_{\alpha}}|}Z\bwvec{\alpha}{\beta}^{\mathfrak{S}_{\alpha}},
\end{align*}
where $\beta$ is an arbitrary composition of length $k$.
\item\mlabel{thmitemalk0s} If $\alpha_k>0$, then for any nonnegative integer $s$ and composition $\beta$ with $\ell(\beta)=k$, we have 
$$Z\bwvec{\alpha,0^s}{\beta,\delta^s}=Z\bwvec{\alpha,0^s}{\delta^{k+s}}.$$
\item\mlabel{thmitemNaldltall} $Z\bwvec{\alpha}{\alpha+\delta}=Z\bwvec{\alpha}{\delta^k}$. 
\end{enumerate}
\end{theorem}
\begin{proof}
\mref{thmitem:skgeq0} This follows directly from Corollary \mref{cor:phi+=phiskneq0}.

\smallskip

\noindent
\mref{thmitem:sk=0}
Since $\mathfrak{S}_{\alpha}$ acts on  $Z\bwvec{\alpha}{\alpha+\delta}$ permuting the last components of ${\alpha}+\delta$ that equal $\delta$,
we obtain
\begin{align*}
Z\bwvec{\alpha}{\alpha+\delta}
=\frac{1}{|{\mathfrak{S}_{\alpha}}|}Z\bwvec{\alpha}{\alpha+\delta}^{\mathfrak{S}_{\alpha}}.
\end{align*}
By Lemma \mref{thm:gzeta}, $Z\bwvec{\alpha}{\beta}^{\mathfrak{S}_{\alpha}}$ is independent of the choice of $\beta$,
so $Z\bwvec{\alpha}{\alpha+\delta}^{\mathfrak{S}_{\alpha}}=Z\bwvec{\alpha}{\beta}^{\mathfrak{S}_{\alpha}}$ and the proof follows.
\smallskip

\noindent
\mref{thmitemalk0s}
By Item \mref{thmitem:sk=0},  we have
\vspace{-.3cm}
\begin{align*}
Z\bwvec{\alpha,0^s}{\beta,\delta^s}
=\frac{1}{|{\mathfrak{S}_{(\alpha,0^s)}}|}Z\bwvec{\alpha,0^s}{\delta^{k+s}}^{\mathfrak{S}_{(\alpha,0^s)}}=Z\bwvec{\alpha,0^s}{\delta^{k+s}}
\end{align*}
since the action of ${\mathfrak{S}_{(\alpha,0^s)}}$ leaves $Z\bwvec{\alpha,0^s}{\delta^{k+s}}$ unchanged.
\smallskip

\noindent
\mref{thmitemNaldltall} Let $\delta$ be an arbitrary positive integer. If $\alpha_k>0$, then,  by Corollary \mref{cor:phi+=phiskneq0}, $Z\bwvec{\alpha}{\alpha+\delta}={M}_{{\alpha}}=Z\bwvec{\alpha}{\delta^k}$.
If $\alpha_k=0$, then by Item \mref{thmitemalk0s}, 
$
Z\bwvec{\alpha}{\alpha+\delta}=Z\bwvec{\alpha}{\delta^k}.
$
\end{proof}

By Theorem \mref{thm:mainthmwildms}, when $\alpha$ is a left weak composition, $Z\bwvec{\alpha}{\delta^k}=Z\bwvec{\alpha}{\alpha+\delta}$ coincides with the monomial left weak quasisymmetric function
$M_{\alpha}$. So we may regard $Z\bwvec{\alpha}{\delta^k}$ as an extension of the notion $M_{\alpha}$, from $\alpha$ being left weak compositions to being weak compositions, justifying the following notion.

\begin{defn}\mlabel{def:mrqf}
The {\bf renormalized monomial quasisymmetric function} of a weak composition $\alpha=(\alpha_1,\cdots,\alpha_k)$ is
\vspace{-.5cm}
\begin{align*}
{M}_{{\alpha}}:=Z\bwvec{\alpha}{\delta^k}.
\end{align*}
A \name{renormalized quasisymmetric function} is a $\bfk$-linear combination of renormalized monomial quasisymmetric functions. Let $\RenQSym$ denote the set of renormalized quasisymmetric functions. 
\end{defn}

\begin{exam}\mlabel{exam:m000}
The following are some examples of renormalized monomial quasisymmetric functions. Let $s$ be in $\PP$.
\begin{align*}
{M}_{0}=-t-\frac{1}{2},\qquad {M}_{(0,0)}=\frac{1}{2}t^2+t+\frac{3}{8}, \qquad {M}_{(0,0,0)}=-\frac{1}{6}t^3-\frac{3}{4}t^2-\frac{23}{24}t-\frac{5}{16},
\end{align*}
\begin{align*}
{M}_{(s,0)}=&-\Big(t+\frac{1}{2}\Big)\widehat{M}_{\wvec{s}{0}}-\widehat{M}_{\wvec{s}{1}} = -(t+\frac{3}{2})M_s-M_{(0,s)},\\
{M}_{(0,s,0)}=&-\Big(t+\frac{1}{2}\Big)\widehat{M}_{\wvec{0,s}{0,0}}-\widehat{M}_{\wvec{0,s}{0,1}}=-(t+\frac{5}{2})M_{(0,s)}-2M_{(0,0,s)}
\end{align*}
\begin{align*}
{M}_{(s,0,0)}=\Big(\frac{t^2}{2}+t+\frac{3}{8}\Big)\widehat{M}_{\wvec{s}{0}}+(t+1)\widehat{M}_{\wvec{s}{1}}+\frac{1}{2}\widehat{M}_{\wvec{s}{2}}
=\Big(\frac{t^2}{2}+2t+\frac{15}{8}\Big) {M}_{s}+(t+\frac{3}{2}) {M}_{(0,s)}+{M}_{(0,0,s)}.
\end{align*}
\end{exam}

Observe that all the renormalized monomial quasisymmetric functions given in Example \ref{exam:m000} are polynomials in $t$ with coefficients in $\LWQSym$.
This turns out to be a general phenomenon, as shown in Theorem \mref{thm:rqflmcasymm0}.

\section {Properties of renormalized quasisymmetric functions}
\mlabel{sec:reln}

In this section, we obtain several properties on the structure of renormalized quasisymmetric functions of weak compositions. We first prove that the product of renormalized  monomial quasisymmetric functions satisfies the quasi-shuffle relation, just like their counter parts for compositions. We then
show that the algebra $\GSym$ coincides with the algebra $\LWCQSym$ of left weak quasisymmetric functions, and the algebra $\RenQSym$ of renormalized quasisymmetric functions
is the polynomial algebra in the variable $M_0$ over the ring $\LWCQSym$.
We further display two linear bases for $\RenQSym$ and apply them to equip renormalized quasisymmetric functions with the structures of a
Hopf algebra and a free commutative Rota-Baxter algebra.

\subsection{Quasi-shuffle relation of renormalized quasisymmetric functions}

We now show that the renormalized monomial quasisymmetric functions satisfy the quasi-shuffle relation.

\begin{theorem}\mlabel{thm:widetildeMsquasishuffle}
Renormalized monomial quasisymmetric functions satisfy the quasi-shuffle relation, that is,
\begin{align*}
{M}_{\alpha}{M}_{\beta}={M}_{\alpha*\beta}
\end{align*}
for any weak compositions $\alpha$ and $\beta$. Therefore the set $\RenQSym$ of all renormalized quasisymmetric functions is a subalgebra of $\bfk[[X]][t]$.
\end{theorem}

\begin {proof} First we verify $ M_{0^m}M_{0^n}=M_{0^m * 0^n}$ by inductions on $n$.
Applying Corollary~\mref{cor:Z(s,s)starrul}, for $m\ge 1$, 
$$M_{0^{m}}M_0=Z\bwvec {0^{m}}{\delta^{m}}Z\bwvec {0}{\delta}=(m+1)Z\bwvec {0^{m+1}}{\delta^{m+1}}+\sum_{\beta} Z\bwvec {0^{m}}{\beta}
$$
where $\beta $ runs over all compositions with $m-1$ components equaling $\delta$ and one component equaling $2\delta$.
By Theorem \mref{thm:mainthmwildms},
$$
\sum_{\beta}  Z\bwvec {0^{m}}{\beta}=\frac{1}{(m-1)!}Z\bwvec {0^{m}}{2\delta,\delta^{m-1}}^{\mathfrak{S}_{0^m}}
=mZ\bwvec {0^{m}}{\delta^{m}}=mM_{0^m}.
$$
So
$$M_{0^{m}}M_0=(m+1)M_{0^{m+1}}+mM_{0^{m}}=M_{0^{m}*0}.
$$
Now let $n\ge 1$ and assume  
$M_{0^{m}}M_{0^{n}}=M_{0^{m}*0^{n}}.
$
Since $0^n*0=(n+1)0^{n+1}+n 0^n$ gives
$$0^{n+1}=\frac{1}{n+1}0^n*0-\frac{n}{n+1} 0^n,$$
we obtain
\begin{align*}
M_{0^{m}}M_{0^{n+1}}&=\frac{1}{n+1}M_{0^{m}}(M_{0^{n}}M_0-nM_{0^n})=\frac{1}{n+1}M_{0^m*0^n}M_{0}-\frac{n}{n+1}M_{0^m*0^n}\\
&=\frac{1}{n+1}M_{0^m*0^n*0}-\frac{n}{n+1}M_{0^m*0^n}=M_{0^m*(\frac{1}{n+1}0^n*0-\frac{n}{n+1}0^n)}=M_{0^{m}*0^{n+1}}.
\end{align*}

Now for $\alpha \in \LWC$, we prove 
$$M_{0^m}M_{(\alpha, 0^n)}=M_{0^m*(\alpha, 0^n)}
$$
by induction on $n\geq 0$. For $n=0$, there exist $\alpha^{(i)}\in \LWC$ and $c_i\in \mathbb{Z}$ such that
$$\wvec {0^m}{\delta ^m}*\wvec {\alpha}{\delta ^{\ell(\alpha)}}=\sum_i c_i\wvec {\alpha^{(i)}, 0^{k_i}}{\beta^{(i)}, \delta ^{k_i}}
$$
and $0^m*\alpha=\sum_i c_i(\alpha^{(i)}, 0^{k_i})$. So by Corollary \mref{cor:Z(s,s)starrul} and Theorem \mref{thm:mainthmwildms}\mref{thmitemalk0s},
\begin{align*}
M_{0^m}M_{\alpha}=Z\bwvec {0^{m}}{\delta^{m}}Z\bwvec {\alpha }{\delta^{\ell(\alpha)}}=\sum_i c_iZ\bwvec {\alpha^{(i)}, 0^{k_i}}{\beta^{(i)}, \delta ^{k_i}}
=\sum_i c_iM_{(\alpha^{(i)}, 0^{k_i})}=M_{0^m*\alpha}.\end{align*} 
For $n\ge 1$, it follows from the definition of quasi-shuffle product $*$ that there exist $\alpha^{(i)}\in \LWC$,  $k_i<n$ and $r_i\in\mathbb{Q}$ such that
$$(\alpha , 0^n)=0^n*\alpha +\sum_i r_i (\alpha^{(i)}, 0^{k_i}).
$$
Then applying the induction hypothesis on $n$ yields
\begin{align*}
M_{0^m}M_{(\alpha, 0^n)}&=M_{0^m}\Big(M_{ 0^n}M_{\alpha}+\sum_i r_i M_{(\alpha^{(i)}, 0^{k_i})}\Big)
=M_{0^m*0^n}M_{\alpha}+\sum_i r_i M_{0^m*(\alpha^{(i)}, 0^{k_i})}\\
&=M_{0^m*0^n*\alpha}+\sum_i r_i M_{0^m*(\alpha^{(i)}, 0^{k_i})}=M_{0^m*(\alpha, 0^n)}.
\end{align*}

Finally for $\alpha , \beta \in \LWC$ and $m, n\geq 0$, we prove 
$$M_{(\alpha, 0^m)}M_{(\beta, 0^n)}=M_{(\alpha, 0^m)*(\beta, 0^n)}
$$
again by induction on $n\geq 0$. For $n=0$, as above we can write
$$\wvec {\alpha, 0^m}{\delta ^{\ell(\alpha)+m}}*\wvec {\beta }{\delta ^{\ell(\beta)}} =\sum_i c_i\wvec {\gamma^{(i)}, 0^{k_i}}{\mu^{(i)}, \delta ^{k_i}},
$$
with $\gamma^{(i)}\in \LWC$ and $(\alpha, 0^m)*\beta=\sum_i c_i(\gamma^{(i)}, 0^{k_i})$. Then, by Corollary \mref{cor:Z(s,s)starrul} and Theorem \mref{thm:mainthmwildms}\mref{thmitemalk0s},
$$M_{(\alpha, 0^m)}M_{\beta }=Z\bwvec {\alpha, 0^{m}}{\delta^{\ell(\alpha)+m}}Z\bwvec {\beta }{\delta^{\ell(\beta)}}
=\sum_i c_iZ\bwvec {\gamma^{(i)}, 0^{k_i}}{\mu^{(i)}, \delta ^{k_i}}=\sum_i c_iM_{(\gamma^{(i)}, 0^{k_i})}=M_{(\alpha, 0^m)*\beta}.
$$
For $n\ge 1$, let
$$(\beta , 0^n)=0^n*\beta +\sum_i r_i (\beta^{(i)}, 0^{k_i})
$$
where $k_i<n$ and $r_i\in \QQ$. Then we inductively obtain
\begin{align*}M_{(\alpha, 0^m)}M_{(\beta, 0^n)}=&M_{(\alpha, 0^m)}\Big(M_{ 0^n}M_{\beta }+\sum_i r_i M_{(\beta^{(i)}, 0^{k_i})}\Big)\\
=&M_{(\alpha, 0^m)*0^n}M_{\beta }+\sum_i r_i M_{(\alpha, 0^m)*(\beta^{(i)}, 0^{k_i})}\\
=&M_{(\alpha, 0^m)*0^n*\beta }+\sum_i r_i M_{(\alpha, 0^m)*(\beta^{(i)}, 0^{k_i})}\\
=&M_{(\alpha, 0^m)*(\beta, 0^n)}.
\end{align*}
This completes the proof.
\end{proof}

\subsection{Bases for the algebra of  renormalized quasisymmetric functions}
We will give two linear basis of $\RenQSym$. We begin by showing that every Stirling left weak quasisymmetric function is a nonnegative integer linear combination of monomial left weak quasisymmetric functions.

Recall that Stirling numbers of the second kind, denoted  $S(m,i)$ with $m,i\in \NN$, are the numbers of partitions of an $m$-set into $i$ blocks.
For any positive integers $m$ and $n$, it is well known~\cite{Sta12} that
\begin{align*}
n^m=\sum_{i=0}^m i!S(m,i)\binom{n}{k}.
\end{align*}
Here $i!S(m,i)$ enumerates the number of surjections from an $m$-set to a $i$-set. Similarly, it is also not hard to see that
\begin{align*}
n^m=\sum_{i=0}^m i!S(m+1,i+1)\binom{n-1}{i},
\end{align*}
noting that $i!S(m+1,i+1)$ is the number of surjections $h$ from the set $[m+1]$ to the set $[i+1]$ such that $h(m+1)=i+1$. It is also a special case of Lemma~\mref{lem:filtst}.
In order to investigate the relationship between \gslwqsyms and monomial left weak quasisymmetric functions, we generalize this number to the setting of multiple sets.

Let $n_1<n_2<\cdots<n_k$ be positive integers and let $\beta=(\beta_1,\beta_2,\cdots,\beta_k)$ be a weak composition. For $1\leq t\leq k$, denote $b_j:=\sum\limits_{t=1}^j \beta_t$. Then $b_k=|\beta|$.
A map
$f:[b_k] \to [n_k]$
is called {\bf filtered} if
\begin{equation*} 
f([b_j]) \subseteq [n_j], \quad
\quad 1\leq j\leq k.
\end{equation*}
The inclusions can be equivalently replaced by
$f([b_{j-1}+1,b_j])\subseteq [n_j]$, where $b_0:=0$.
Let $T$ denote the set of filtered maps $f:[b_k]\to [n_k]$. Then by this equivalent condition, we have
$$\# T=n_1^{\beta_1}n_2^{\beta_2}\cdots n_k^{\beta_k}.$$
Here if $\beta=0^k$, then $[b_k]$ is an empty set. We assume that $T$ contains the unique empty map.

Let $\{u_1,\cdots,u_k\}$ be a set disjoint from $[b_k]$.
A map
$$\bar{f}:[b_k]\cup \{u_1,\cdots,u_k\} \to [n_k]$$
is called a \name{filtered pointed map} if
\begin{equation} \mlabel{eq:filf}
\bar{f}([b_j]) \subseteq [n_j], \quad
\bar{f}(u_j)=n_j, \quad 1\leq j\leq k.
\end{equation}
Let $\overline{T}$ denote the set of filtered pointed maps. For any $I=(i_1,\cdots,i_k)$ and any subsets $Y_j$ of $[n_{j-1}+1,n_j-1]$ with cardinality $i_j, 1\leq j\leq k$, denote 
\begin{equation}\mlabel{eq:cstir}
	c_{\beta,I}=\#\{\text{filtered pointed maps }\bar{f} \text{ such that } \mim \bar{f} \cap [n_{j-1}+1,n_j-1]=Y_j, 1\leq j\leq k\}.
\end{equation}

\begin{remark}\mlabel{rk:cstir}
	In the special case of $k=1, \beta_1=m$ and $I=\{i\}, 0\leq i\leq k$, we have
$$
	c_{m,i}:=c_{\beta_1,(i)}=i! S(\beta_1+1,i+1).
$$
Thus we may regard $c_{\beta,I}$ as a filtered pointed variation of Stirling numbers of the second kind.
\end{remark}
\begin{lemma} \mlabel{lem:filtst}
We have
\begin{equation} \mlabel{eq:cbi}
n_1^{\beta_1}n_2^{\beta_2}\cdots n_k^{\beta_k}=\sum_{I=(i_1,i_2,\cdots,i_k)} c_{\beta,I} \binom{n_1-1}{i_1}\binom{n_2-n_1-1}{i_2} \cdots\binom{n_k-n_{k-1}-1}{i_k},
\end{equation}
where $i_j\in \NN$ is such that $\sum\limits_{t=j}^ki_t\leq \sum\limits_{t=j}^k\beta_t$ for $j=1,2,\cdots,k$. Here we take the convention that $c_{0^k,I}=1$ and
$\binom{m}{n}=0$ if $m<n$.
\end{lemma}

\begin{proof}
First note that we have a bijection
$$ \overline{T} \rightarrow T, \qquad\bar{f}\mapsto \bar{f}\Big|_{[b_k]},$$
with the inverse map to be extending an $f\in T$ to a $\bar{f}$ by assigning $\bar{f}(u_j):=n_j, 1\leq j\leq k.$
Therefore,
\begin{equation} \mlabel{eq:tt'}
\# \overline{T}=\# T= n_1^{\beta_1}\cdots n_k^{\beta_k}.
\end{equation}

Next for any filtered pointed map $\bar{f}\in \overline{T}$ with the corresponding $f:=\bar{f}\Big|_{[n_k]}$, denote 
$$Y_j:=\mim \bar{f} \cap [n_{j-1}+1,n_j-1], \quad i_j:=\# Y_j.$$
By Eq.~\meqref{eq:filf}, for each $j$ with $1\leq j\leq k$, we have
 $f^{-1}(Y_j\cup \cdots\cup Y_k) \subseteq [b_{j-1}+1,b_k]$. Note that $Y_j\cup \cdots\cup Y_k$ is a disjoint union, so we have
$\sum\limits_{t=j}^k i_t \leq \sum\limits_{t=j}^k \beta_t, 1\leq j\leq k$.

Let such a collection of subsets $\bar{Y}_j, 1\leq j\leq k,$ be given.
Then $c_{\beta,I}$ is the number of filtered pointed maps with $\mim \bar{f} \cap [n_{j-1}+1,n_j-1] = \bar{Y}_j, 1\leq j\leq k$. Since there are $\bincc{n_j-n_{j-1}-1}{i_j}$ choices of the subsets of $[n_{j-1}+1,n_j-1]$ with the same cardinalities, we obtain
$$ \# \overline{T} = \sum_{I=(i_1,i_2,\cdots,i_k)} c_{\beta,I} \binom{n_1-1}{i_1}\binom{n_2-n_1-1}{i_2} \cdots\binom{n_k-n_{k-1}-1}{i_k},$$
where $i_j\in \NN$ is such that $\sum\limits_{t=j}^ki_t\leq \sum\limits_{t=j}^k\beta_t$ for $j=1,2,\cdots,k$.
Combining with Eq.~\meqref{eq:tt'}, this completes the proof.
\end{proof}

\begin{prop}\mlabel{prop:transfMtohatM}
With the notation $c_{\beta,I}$ in Eq.~\meqref{eq:cstir}, for any \ulwb
$\wvec{\alpha}{ \beta}=\wvec{\alpha_1,\alpha_2,\cdots,\alpha_k}{\beta_1,\beta_2,\cdots,\beta_k},$
the corresponding \gslwqsym in Eq.~\meqref{eq:sqsym1} is given by
$$\widehat{M}_{\wvec{\alpha}{\beta}}=\sum_{I=(i_1,i_2,\cdots,i_k)}c_{\beta,I}{M}_{(0^{i_1},\alpha_1,0^{i_2},\alpha_2,\cdots,0^{i_k},\alpha_k)},$$
where $0\leq i_j\leq \beta_j+\beta_{j+1}+\cdots+\beta_k$, $j=1,2,\cdots,k$.
\end{prop}
\begin{proof}
It follows from  Eqs.~\eqref{eq:sqsym1} and \eqref{eq:cbi} that
\begin{align*}
\widehat{M}_{\wvec{\alpha}{\beta}}
=&\sum_{I=(i_1,i_2,\cdots,i_k)}c_{\beta,I}\sum_{n_1<n_2<\cdots<n_k}\binom{n_1-1}{i_1}\binom{n_2-n_1-1}{i_2}\binom{n_k-n_{k-1}-1}{i_k}x_{n_1}^{\alpha_1}x_{n_2}^{\alpha_2}\cdots x_{n_k}^{\alpha_k}.
\end{align*}
By Eq.~\meqref{eq:malphamonomiabasis}, for the nonnegative integers $i_1,i_2,\cdots,i_k$, we have
\begin{align*}
{M}_{(0^{i_1},\alpha_1,0^{i_2},\alpha_2,\cdots,0^{i_k},\alpha_k)}
=&\sum x_{m_{1}}^0\cdots x_{m_{i_1}}^0 x_{n_1}^{\alpha_1}x_{m_{i_1+1}}^0\cdots x_{m_{i_1+i_2}}^0 x_{n_2}^{\alpha_2} \cdots x_{m_{i_1+\cdots+i_{k-1}+1}}^{0} \cdots x_{m_{i_1+\cdots+i_k}}^{0}  x_{n_k}^{\alpha_k}\\
=&\sum_{n_1<n_2<\cdots<n_k}\binom{n_1-1}{i_1}\binom{n_2-n_1-1}{i_2}\cdots\binom{n_k-n_{k-1}-1}{i_k}x_{n_1}^{\alpha_1}x_{n_2}^{\alpha_2}\cdots x_{n_k}^{\alpha_k},
\end{align*}
where the first summation is over all positive integer sequences
$$m_{1}<\cdots<m_{i_1}< n_1<m_{i_1+1}<\cdots<m_{i_1+i_2}<n_2<\cdots<m_{i_1+\cdots+i_{k-1}+1}<\cdots<m_{i_1+\cdots+i_k}<n_k.$$
Then the proof follows.
\end{proof}

\begin{exam} By Proposition \mref{prop:transfMtohatM}, we have
\begin{align*}
\widehat{M}_{\wvec{s}{1}}=M_{(0,s)}+M_{(s)}\ \ {\rm and }\ \
\widehat{M}_{\wvec{s}{2}}={M}_{(s)}+3{M}_{(0,s)}+2{M}_{(0,0,s)}.
\end{align*}
We also have $\widehat{M}_{\wvec{0,s}{0,1}}=2M_{(0,s)}+2M_{(0,0,s)}$ and
\begin{align*}
\widehat{M}_{\wvec{s_1,s_2}{1,1}}
=M_{(0,s_1,0,s_2)}+2M_{(0,0,s_1,s_2)}+4M_{(0,s_1,s_2)}+2M_{(s_1,s_2)}+M_{(s_1,0,s_2)}.
\end{align*}
\end{exam}

\begin{prop}\mlabel{thm:zqsym=lwqsym}
We have $\GSym= \LWCQSym$.
\end{prop}
\begin{proof}
Note that
Proposition \mref{prop:transfMtohatM} yields $ \GSym\subseteq \LWCQSym$. On the other hand, for any left weak composition $\alpha$,
it is easy to see that
$M_{\alpha}=\widehat{M}_{\wvec{\alpha}{0^{\ell(\alpha)}}}\in  \GSym$.
Thus $ \GSym= \LWCQSym$.
\end{proof}

\begin{theorem}\mlabel{thm:rqflmcasymm0}
$\LWCQSym[t]$ coincides with the algebra $\RenQSym$ of renormalized quasisymmetric functions.
\end{theorem}

\begin{proof}
By Theorem \mref{thm:widetildeMsquasishuffle}, $\RenQSym$ is a subalgebra algebra of $\bfk[[X]][t]$.
Moreover,
$\RenQSym\subseteq \GSym[t]$ according to Definitions \mref{defn:zetaqsf} and \mref{def:mrqf}.
It follows directly from Theorem \mref{thm:mainthmwildms}\mref{thmitem:skgeq0} that $\LWCQSym \subseteq \RenQSym$, which together with
the fact that ${M}_{0}=-t-\frac{1}{2}$, as showed by Example \mref{exam:m000}, yields the inclusion $\LWCQSym[t] \subseteq \RenQSym$. Hence we have  $\RenQSym=\LWCQSym[t]$ by Proposition \mref{thm:zqsym=lwqsym}.
\end{proof}

\begin{coro}\mlabel{thm:twobasesLQWSym}
The set $\Big\{{M}_{0}^n{M}_{\alpha}\,\Big|\,n\in\NN,\alpha\in \LWC\Big\}$ is a $\bfk$-basis for $\RenQSym$.
\end{coro}
\begin{proof}
Since $M_0=-t-\frac{1}{2}$,  $M_0$ is algebraically independent over $\LWCQSym$.
Then the statement immediate follows from Theorem~\mref{thm:rqflmcasymm0} and the facts that
$\{{M}_{0}^n\,|\,n\in\NN\}$ is a $\bfk$-basis for $\bfk[t]$ and $\{{M}_{\alpha}|\alpha\in \LWC\}$ is a $\bfk$-basis for $\LWCQSym$.
\end{proof}

We next show that $\{{M}_{\alpha}\,|\,\alpha\in WC\}$ is also a $\bfk$-basis for $\RenQSym=\LWCQSym[t]$.
First note that the quasi-shuffle relation of $M_\alpha$ in Theorem~\mref{thm:widetildeMsquasishuffle} gives
\begin{align*}
{M}_{0}{M}_{0^k}=(k+1){M}_{0^{k+1}}+k{M}_{0^k}.
\end{align*}
This is precisely the recurrence relation characterizing the the divided falling factorials $(x)_n/n!$ with $(x)_n:=x(x-1)\cdots (x-n+1), n\geq 0,$ the falling factorials~\mcite{Co}. Thus we obtain

\begin{lemma}\mlabel{lem:N0k=ak}
For any positive integer $k$, we have 
\begin{align*}
{M}_{0^k}=\frac{\prod_{i=0}^{k-1}({M}_{0}-i)}{k!}=\frac{(-1)^k}{k!}\prod_{i=1}^k\Big(t+i-\frac{1}{2}\Big).
\end{align*}
\end{lemma}

We now give an explicit expression of ${M}_{\alpha}$
in terms of monomial left weak quasisymmetric functions.

\begin{lemma}\mlabel{lemNalp=polyofa}
Let $\alpha=(\alpha_1,\cdots,\alpha_j,0^{k_\alpha})$ be a weak composition, where $\alpha_j\in\PP$ and $k_{\alpha}\in\NN$. Denote $\alpha'=(\alpha_1,\cdots,\alpha_{j-1})$. Then
\begin{align*}
{M}_{\alpha}=\sum_{p=0}^{k_{\alpha}}\frac{(-1)^{p}\prod_{i=1}^{k_\alpha-p}({M}_{0}-\ell(\alpha)+i)}{(k_\alpha-p)!}M_{(\alpha'\shap0^p,\alpha_j)}
=(-1)^{k_{\alpha}}\sum_{p=0}^{k_{\alpha}}\frac{\prod_{i=1}^{k_\alpha-p}(t+\ell(\alpha)-i+\frac12)}{(k_\alpha-p)!}M_{(\alpha'\shap0^p,\alpha_j)}. 
\end{align*}
\end{lemma}
\begin{proof}
We prove by induction on $k_{\alpha}$. Clearly, the assertion is true for $k_\alpha=0$.
Now assume that the first desired identity holds for all weak compositions $\beta$ with $k_{\beta}<k_{\alpha}$.
By Theorem \mref{thm:widetildeMsquasishuffle},
\begin{align*}
{M}_{0}{M}_{(\alpha',\alpha_j,0^{k_{\alpha}-1})}
={M}_{0*(\alpha',\alpha_j,0^{k_{\alpha}-1})}
={M}_{(\alpha'\shap0,\alpha_j,0^{k_{\alpha}-1})}+k_{\alpha}{M}_{\alpha}+(\ell(\alpha)-1){M}_{(\alpha',\alpha_j,0^{k_{\alpha}-1})},
\end{align*}
and hence
\vspace{-.3cm}
\begin{align*}
{M}_{\alpha}=\frac{{M}_{0}-\ell(\alpha)+1}{k_{\alpha}}
{M}_{(\alpha',\alpha_j,0^{k_{\alpha}-1})}-\frac{1}{k_{\alpha}} {M}_{(\alpha'\shap0,\alpha_j,0^{k_{\alpha}-1})}.
\end{align*}
It follows from the induction hypothesis that
\begin{align*}
{M}_{\alpha}
=&\frac{1}{k_{\alpha}}\sum_{p=0}^{k_{\alpha}-1}\frac{(-1)^{p}\prod_{i=1}^{k_\alpha-p}({M}_{0}-\ell(\alpha)+i)}{(k_\alpha-1-p)!}M_{(\alpha'\shap0^p,\alpha_j)}\\
&\qquad\qquad-\frac{1}{k_{\alpha}} \sum_{p=0}^{k_{\alpha}-1}\frac{(-1)^{p}\prod_{i=1}^{k_\alpha-1-p}({M}_{0}-\ell(\alpha)+i)}{(k_\alpha-1-p)!}M_{(\alpha'\shap0\shap0^p,\alpha_j)}\\
=&\frac{1}{k_{\alpha}}\sum_{p=0}^{k_{\alpha}-1}\frac{(-1)^{p}\prod_{i=1}^{k_\alpha-p}({M}_{0}-\ell(\alpha)+i)}{(k_\alpha-1-p)!}M_{(\alpha'\shap0^p,\alpha_j)}\\
&\qquad\qquad+\frac{1}{k_{\alpha}} \sum_{p=1}^{k_{\alpha}}\frac{(-1)^{p}\prod_{i=1}^{k_\alpha-p}({M}_{0}-\ell(\alpha)+i)}{(k_\alpha-p)!}pM_{(\alpha'\shap0^{p+1},\alpha_j)},
\end{align*}
which yields the the first equality. The second equality follows since ${M}_{0}=-t-1/2$.
\end{proof}

\begin{theorem}\mlabel{prop:nalpisakbasis}
\begin{enumerate}
	\item \mlabel{it:iso1}
The set $\{{M}_{\alpha}|\alpha\in WC\}$ is a $\bfk$-basis for $\RenQSym$.
\item \mlabel{it:iso2}
The algebras $\RenQSym$ and $\QS(\NN)$ are isomorphic. 
\end{enumerate}
\end{theorem}

\begin{proof}
\mref{it:iso1}
We first show that the set $\{{M}_{\alpha}|\alpha\in WC\}$ is $\bfk$-linearly independent. Assume that $\sum\limits_{\alpha\in \Lambda}c_\alpha {M}_{\alpha}=0$ where $\Lambda$ is a nonempty finite set of weak compositions, and $c_{\alpha}\in \bfk$ is nonzero for all $\alpha\in \Lambda$.
For each weak composition $\alpha$, we can uniquely write $\alpha=(\alpha{''}, 0^{k_\alpha})$, where $\alpha{''}$ is a left weak composition and $k_{\alpha}$ is a  nonnegative integer.
Then, by Lemmas \mref{lem:N0k=ak} and \mref{lemNalp=polyofa},
\begin{align}\mlabel{eq:Na=(-1)kallakal}
{M}_{\alpha}=\frac{1}{k_{\alpha}!}{M}_{0}^{k_{\alpha}}M_{\alpha{''}}+(\text{terms of lower degree in ${M}_{0}$}).
\end{align}
Here we regard ${M}_{\alpha}$ as a polynomial in ${M}_{0}$ with coefficients in $\LWCQSym$ since $M_0=-t-1/2$ is also a polynomial generator of $\LWCQSym[t]$ and we write $\LWCQSym[t]=\LWCQSym[M_0]$ when an element is expressed as a polynomial in $M_0$. Thus,
\begin{align*}
\sum_{\alpha\in \Lambda}c_\alpha {M}_{\alpha}=\sum_{\alpha\in \Lambda}\frac{1}{k_{\alpha}!}c_\alpha {M}_{0}^{k_{\alpha}} M_{\alpha{''}}+g({M}_{0}),
\end{align*}
where $g({M}_{0})$ is a polynomial in $\LWCQSym[{M}_{0}]$.
Let $k=\max\{k_\alpha\,|\,\alpha\in \Lambda\}$ and let $\Lambda_1=\{\alpha\in\Lambda\,|\,k_\alpha=k\}$.
Then $\Lambda_1$ is nonempty since $\Lambda$ is a nonempty finite set. Hence,
\begin{align*}
\sum_{\alpha\in \Lambda}c_\alpha {M}_{\alpha}=\sum_{\alpha\in \Lambda_1}\frac{1}{k!}c_\alpha {M}_{0}^{k}M_{\alpha{''}}+(\text{terms of lower degree in ${M}_{0}$}).
\end{align*}
Therefore, $\sum\limits_{\alpha\in \Lambda}c_\alpha {M}_{\alpha}=0$  implies that
$
\sum\limits_{\alpha\in \Lambda_1}\frac{1}{k!}c_\alpha {M}_{0}^{k}M_{\alpha{''}}=0.
$
By Theorem \mref{thm:twobasesLQWSym},
$\{{M}_{0}^k{M}_{\alpha}\,|\,k\in\NN,\alpha\in \LWC\}$ is a $\bfk$-basis for $\LWCQSym[{M}_{0}]$. Hence $c_{\alpha}=0$ for all $\alpha\in \Lambda_1$,
contradicting the assumption that $c_{\alpha}\neq0$ for all $\alpha\in \Lambda$ and $\Lambda_1\subseteq \Lambda$.
Therefore,  $\{{M}_{\alpha}|\alpha\in WC\}$ is $\bfk$-linearly independent.

We next show that for any $k\in\NN$ and $\alpha\in \LWC$, the element ${M}_{0}^kM_{\alpha}$ and hence $t^kM_\alpha$ is a $\bfk$-linear combination of $\{{M}_{\alpha}\,|\,\alpha\in WC\}$, showing that the set $\{{M}_{\alpha}|\alpha\in WC\}$ is a spanning set of $\RenQSym=\LWCQSym[t]$.
We prove the statement by induction on $k$. If $k=0$, then ${M}_{0}^kM_{\alpha}=M_{\alpha}$ is in $\{{M}_{\alpha}\,|\,\alpha\in WC\}$, since $\alpha$ is a left weak composition.
Now assume that the statement holds for all nonnegative integers less than a given positive integer $k$. Then it follows from Eq.~\meqref{eq:Na=(-1)kallakal} that
\begin{align*}
{M}_{0}^kM_{\alpha}={k!}{M}_{(\alpha,0^k)}+(\text{terms of lower degree in ${M}_{0}$}).
\end{align*}
By the induction hypothesis, terms with degree in ${M}_{0}$ lower than $k$ are $\bfk$-linear combinations of $\{{M}_{\alpha}\,|\,\alpha\in WC\}$, so this completes the induction.
\smallskip

\noindent
\mref{it:iso2}
By Item~\mref{it:iso1}, $\RenQSym=\LWCQSym[t]$ is a $\bfk$-algebra with linear basis $\{{M}_{\alpha}|\alpha\in \WC\}$. Thus by Theorem \mref{thm:widetildeMsquasishuffle},
the map
\begin{equation} \mlabel{eq:wciso}
	f: \QS(\NN)=\bfk \WC \rightarrow \RenQSym=\LWQSym[t],\qquad f(\alpha)=M_\alpha, \quad \alpha\in \WC,
\end{equation}
is an algebra isomorphism.
\end{proof}

\subsection{Hopf algebras and Rota-Baxter algebras}
We end the paper with some direct consequences of renormalized quasisymmetric functions on Hopf algebras and Rota-Baxter algebras. 

First Theorem~\mref{prop:nalpisakbasis} immediately equips the algebra of renormalized quasisymmetric functions of weak compositions with a Hopf algebra structure and a Rota-Baxter algebra structures.

By transporting of structures, the isomorphism $f$ in Theorem~\mref{prop:nalpisakbasis}\mref{it:iso2} 
induces a Hopf algebra structure on $\RenQSym$ from the one on $\bfk \WC$ recalled in Section~\mref{ss:alg}. To make the induced operations precise, we recall that, for a weak composition $\alpha=(\alpha_1,\cdots,\alpha_k)$ and a composition $J=(j_1,j_2,\cdots,j_\ell)$ of $k$, the weak composition $J[\alpha]$ of $\alpha$ is defined by
\begin{align*}
J[\alpha]:=(\alpha_1+\cdots+\alpha_{j_1},\alpha_{j_1+1}+\cdots+\alpha_{j_1+j_2},\cdots,\alpha_{j_1+j_2+\cdots+j_{\ell-1}+1}+\cdots+\alpha_{k}).
\end{align*}
Then the coproduct, counit  and antipode on $\LWCQSym[t]$ are given by
\begin{align*}
    \Delta_{L}({M}_{\alpha}):=\sum_{\alpha=\beta\cdot\gamma}{M}_{\beta}\otimes {M}_{\gamma}, \quad \epsilon_{L}(N_\alpha):=\delta_{\alpha,\emptyset}\quad {\rm and} \quad
    S_{L}(N_\alpha):=(-1)^{\ell(\alpha)}\sum_{J\models \ell(\alpha)}{M}_{J[\alpha^r]}.
\end{align*}

Next the divergence of monomial quasisymmetric functions for weak compositions was addressed by a formal regularization in a previous work~\mcite{GTY19}. See~\mcite{GTY20} for connection with the Malvenuto-Reutenauer Hopf algebra~\mcite{MR}. 
The resulting Hopf algebra $\QSym_{\tilde\NN}$ of weak composition quasisymmetric function is also isomorphic to $\bfk \WC$. Thus $\RQSym$ is isomorphic to $\RenQSym$ as Hopf algebras, showing the consistence of the two constructions. As noted in the introduction, one advantage of the present construction is its concrete form as power series. 
On the other hand, this Hopf algebra isomorphism and Theorem~3.8 of \mcite{GTY19} directly gives

\begin{prop}\mlabel{pp:subquot} The Hopf algebra $\RenQSym$ has the Hopf algebra $\QSym$ of quasisymmetric functions as both a Hopf subalgebra and a Hopf quotient algebra.
\end{prop}

We finally note that the semigroup algebra $\bfk \NN$ is naturally isomorphic to the polynomial algebra $\bfk[x]$ by sending $i\in \NN$ to $x^i$. Thus the quasi-shuffle algebra $\QS(\NN)$ is naturally isomorphic to the mixable shuffle algebra $\sha(\bfk[x])^+=\oplus_{k\geq 0} (\bfk[x])^{\ot k}$ in~\mcite{GK} by sending $\alpha=(\alpha_1,\cdots,\alpha_k)$ to $x^{\ot \alpha}:=x^{\alpha_1}\ot \cdots \ot x^{\alpha_k}$.
In that paper, the mixable shuffle algebra is used in the construction of the free commutative Rota-Baxter algebra $\sha(x)$ generated by the algebra $\bfk[x]$ or, alternatively, generated by $x$. In fact, by construction,
$$ \sha(\bfk[x]):= \bfk[x]\ot \sha(\bfk[x])^+$$
and the Rota-Baxter operator $P$ on $\sha(\bfk[x])$ is defined by
\begin{equation}
P(x^i\ot x^{\ot \alpha}):=P(x^i\ot x^{\alpha_1}\ot \cdots \ot x^{\alpha_k}): = 1\ot x^{\ot (i)\cdot \alpha}
= 1\ot (x^i\ot x^{\alpha_1}\ot \cdots \ot x^{\alpha_k}).
\mlabel{eq:frbo}
\end{equation}

Then from the algebra isomorphism in Eq.~\meqref{eq:wciso}, we obtain an algebra isomorphism
$$ \sha(\bfk[x])\cong \bfk[x]\ot \LWCQSym[t] =\LWCQSym[t,x]\subseteq \QQ[[X]][t,x].$$
This isomorphism on one hand gives a power series realization of the abstractly defined free Rota-Baxter algebra $\sha(\bfk[x])$ and therefore the corresponding free Rota-Baxter algebras constructed by Rota~\mcite{Ro1} and Cartier~\mcite{Car1972}. On the other hand, the polynomial algebra $\LWCQSym[t,x]$ is equipped with a Rota-Baxter algebra structure from $\sha(\bfk[x])$. Explicitly, by Theorem~\mref{prop:nalpisakbasis}, a linear basis of $\LWCQSym[t,x]=\LWCQSym[M_0,x]$ is given by
$\{x^nM_\alpha\,|\, n\in \NN, \alpha\in \WC\}$. Thus the Rota-Baxter operator $P$ on $\sha(\bfk[x])$ defined in Eq.~\meqref{eq:frbo} induces a Rota-Baxter operator on $\LWCQSym[M_0,x]$ sending $x^nM_\alpha$ to $M_{(n)\cdot\alpha}$.

\smallskip

\noindent {\bf Acknowledgments}: This work was supported by the National Natural Science Foundation of China (Grant No.\@ 11771190, 11821001, 11890663, 12071377, 12071383).

\end{document}